\documentclass[12pt,a4paper]{article}
\usepackage{oldgerm}
\usepackage{amssymb,amsmath,amsthm,graphics,amscd,amsfonts}
\newtheorem{thm}{Theorem}[section]
\newtheorem{prop}[thm]{Proposition}
\newtheorem{lem}[thm]{Lemma}
\newtheorem{cor}[thm]{Corollary}
\newtheorem{definition}[thm]{Definition}

\newcommand{\hK}{hyper-K\"ahler\ }
\newcommand{\HK}{Hyper-K\"ahler\ }
\newcommand{\ImH}{{\rm Im}\mathbb{H}}
\newcommand{\quater}{\mathbb{H}}
\newcommand{\Z}{\mathbb{Z}}
\newcommand{\R}{\mathbb{R}}
\newcommand{\C}{\mathbb{C}}
\newcommand{\hL}{holomorphic\ Lagrangian\ }
\newcommand{\sL}{special\ Lagrangian\ }
\begin{document}
\title{New examples of compact special Lagrangian submanifolds
embedded  
in hyper-K\"ahler manifolds}
\author{Kota Hattori}
\date{}
\maketitle
{\abstract We construct smooth families of 
compact special Lagrangian submanifolds embedded in some toric \hK 
manifolds, which never become \hL submanifolds via any \hK rotations.
These families converge to \sL immersions with 
self-intersection points in the sense of current.
To construct them, we apply the desingularization method developed by Joyce.}
\section{Introduction}
In $1982$, Harvey and Lawson have introduced in \cite{harvey1982calibrated} 
the notion of calibrated 
submanifolds in Riemannian manifold. 
They were recognized by many researchers 
as the important class of minimal submanifolds 
which had already been well-studied for a long time. 
One of the importances of calibrated 
submanifolds is the volume minimizing property, 
that is, 
every compact calibrated submanifold minimizes the 
volume functional in its homology class.

The several kinds of calibrated submanifolds are defined 
in the Riemannian manifolds with special holonomy.
For example, special Lagrangian submanifolds are the 
middle dimensional calibrated submanifolds 
embedded in Riemannian manifolds with $SU(n)$ holonomy, 
so called Calabi-Yau manifolds. 
In \hK manifolds, which are Riemannian manifolds with $Sp(n)$ holonomy, 
there is a notion of \hL submanifolds those are calibrated by the $n$-th power of the K\"ahler form.
At the same time, \hK manifolds are naturally regarded as Calabi-Yau manifolds, 
\sL submanifolds also make sense in these manifolds. 
Hence there are two kinds of calibrated submanifolds 
in \hK manifolds, 
and it is well-known that every \hL submanifold becomes 
\sL by the \hK rotations.
The converse may not holds although the counterexamples have not been found. 

Another importance of calibrated geometry is that some of the calibrated 
submanifolds have the moduli spaces with good structure. 
For instance, 
McLean has shown that the moduli space of 
compact special Lagrangian submanifolds becomes a smooth manifold, 
whose dimension is equal to the first betti number of the 
special Lagrangian submanifold \cite{mclean1996deformations}.

Although the construction of compact \sL submanifolds 
embedded in Calabi-Yau manifolds is not 
easy in general,
Y. I. Lee \cite{lee2003embedded}, Joyce \cite{joyce2003special}\cite{joyce2004special}
and D. A. Lee \cite{lee2004connected} 
developed the gluing method for the construction 
of families of compact \sL submanifolds converging to \sL immersions 
with self-intersection points as a sense of current.
Moreover D. A. Lee construct a non-totally geodesic \sL submanifold 
in the flat torus by applying his gluing method.
After these working, several concrete examples of \sL submanifolds are 
constructed by gluing method.
See \cite{haskins2007slag}\cite{chan2009desingularization1}\cite{chan2009desingularization2},
for example.

In this paper we apply the result in \cite{joyce2003special}\cite{joyce2004special} to the 
construction of new examples of compact \sL submanifolds 
embedded in toric \hK manifolds. 
Moreover, these examples never become \hL submanifolds 
with respect to any complex structures given by the 
\hK rotations.

A \hK manifold is a Riemannian manifold $(M^{4n},g)$ 
equipped with an integrable hypercomplex structure $(I_1,I_2,I_3)$, 
so that $g$ is hermitian with respect to every $I_\alpha$, 
and $\omega_\alpha:=g(I_\alpha\cdot,\cdot)$ are closed. 
For any $\theta\in \R$, note that $e^{\sqrt{-1}\theta}(\omega_2+\sqrt{-1}\omega_3)$ becomes 
a holomorphic symplectic $2$-form with respect to $I_1$.
If the holomorphic symplectic form vanishes on 
a submanifold $L^{2n}\subset M$, 
$L$ is called a \hL submanifold.
Clearly, this definition is not depend on $\theta$.

Similarly, we can define the notion of 
\hL submanifold with respect to a complex structure $aI_1+bI_2+cI_3$ 
for every unit vector $(a,b,c)$ in $\R^3$.
The new complex structure $aI_1+bI_2+cI_3$ is called a \hK rotation 
of $(M,g,I_1,I_2,I_3)$.

The \hK manifold $M$ is naturally regarded as 
the Calabi-Yau manifold by 
the complex structure $I_1$, 
the K\"ahler form $\omega_1$ and the holomorphic volume form 
$(\omega_2 + \sqrt{-1}\omega_3)^n$. 
Then we can easy to see that 
\hL submanifolds with respect to $\cos(\alpha\pi/n) I_2 + \sin(\alpha\pi/n) I_3$ are 
\sL for every $\alpha =1,\cdots, 2n$.
Conversely, it has been unknown whether 
there exist \sL submanifolds embedded in \hK manifolds 
never come from \hL submanifolds with respect to any complex structure 
given by the \hK rotations.
The main result of this paper is described as follows.

\begin{thm}
Let $n\ge 2$. 
There exist smooth compact \sL submanifolds
$\{ \tilde{L}_t\}_{0<t< \delta}$ 
and $\{ L_\alpha\}_{\alpha = 1,\cdots,2n}$ 
embedded in a \hK manifold $M^{4n}$,
which satisfy $\lim_{t\to 0}\tilde{L}_t = \bigcup_\alpha L_\alpha$ 
in the sense of current, and $\tilde{L}_t$ is diffeomorphic to 
$2n(\mathbb{P}^1)^n \# (S^1\times S^{2n-1})$.
Moreover, each $L_\alpha$ is the \hL submanifold of $M$
with respect to $\cos(\alpha\pi/n) I_2 + \sin(\alpha\pi/n) I_3$, 
although $\tilde{L}_t$ never become \hL submanifolds 
with respect to any complex structure 
given by the \hK rotations.
\label{main1}
\end{thm}

This is one of examples which we obtain in this article.
Furthermore, we obtain 
\sL $2\mathbb{P}^2 \# 2\overline{\mathbb{P}^2} \# (S^1\times S^3)$ 
embedded in an $8$-dimensional \hK manifold and 
\sL $(3N+1)(\mathbb{P}^1)^2 \# N(S^1\times S^3)$ 
embedded in another $8$-dimensional \hK manifold, 
both of which never become \hL submanifolds 
with respect to any complex structure 
given by the \hK rotations.

Theorem \ref{main1} has another significance 
from the point of the view of the compactification 
of the moduli spaces of compact \sL submanifolds. 
In general, the moduli space $\mathcal{M}(L)$ of 
the deformations of compact \sL submanifolds $L\subset X$ 
is not necessarily to be compact, 
consequently the study of its compactification 
is important problem.
It is known that the compactification of 
$\mathcal{M}(L)$ is given by 
the geometric measure theory.
The \sL immersion $\bigcup_\alpha L_\alpha$ 
appeared in Theorem \ref{main1} is the concrete example 
of an element of $\overline{\mathcal{M}(\tilde{L}_{t_0})}\backslash\mathcal{M}(\tilde{L}_{t_0})$.
D. A. Lee also considered the similar situation, 
however the Calabi-Yau structures of ambient space of $\tilde{L}_t$ 
is deformed by the parameter $t$ in \cite{lee2004connected}.

Here, we describe the outline of the proof.
Let $(M,J,\omega,\Omega)$ be a K\"ahler manifold 
of complex dimension $m \ge 3$ with 
holomorphic volume form $\Omega\in H^0(K_M)$, 
and $L_{\alpha}\subset M$ be connected \sL submanifolds, 
where $\alpha = 1,\cdots,A$.
Put $\mathcal{V}=\{1,\cdots, A\}$, and suppose we have a 
quiver $(\mathcal{V},\mathcal{E},s,t)$, namely, 
$\mathcal{V}$ consists of finite vertices, 
$\mathcal{E}$ consists of finite directed edges, 
and $s,t$ are maps $\mathcal{E} \to \mathcal{V}$ so that 
$s(h)$ is the source of $h\in \mathcal{E}$ and $t(h)$ is the target.

A subset $S\subset \mathcal{E}$ is called a cycle 
if it is written as $S=\{ h_1,h_2,\cdots,h_l\}$ and $t(h_k) = s(h_{k+1})$, 
$t(h_l) = s(h_1)$ hold for all $k=1,\cdots,l-1$.
Then $\mathcal{E}$ is said to be {\it covered by cycles} if 
every edge $h\in \mathcal{E}$ is contained in some cycles of $\mathcal{E}$.

If there are two \sL submanifolds $L_0,L_1\subset X$ intersecting 
transversely at $p\in L_0\cap L_1$,
then we can define a type at the intersection point $p$, 
which is a positive integer less than $m$.
Then we have the next result, which follows 
from Theorem 9.7 of \cite{joyce2003special} 
by some additional arguments.

\begin{thm}
Let $(\mathcal{V},\mathcal{E},s,t)$ be a quiver, 
and $L_\alpha$ be connected compact \sL submanifolds 
embedded in a Calabi-Yau manifold $(M,J,\omega,\Omega)$ for 
every $\alpha\in \mathcal{V}$. 
Assume that $L_{s(h)}$ and $L_{t(h)}$ intersects transversely 
at only one point $p$ if $h \in \mathcal{E}$, 
and $p$ is the intersection point of type $1$, 
and $L_\alpha \cap L_\beta$ is empty if $\alpha\neq \beta$ and 
there are no edges connecting $\alpha$ and $\beta$.
Then, if $\mathcal{E}$ is covered by cycles, 
there exists a family of 
compact \sL submanifolds $\{ \tilde{L}_t\}_{0<t<\delta}$ embedded in 
$M$ which satisfies $\lim_{t\to 0}\tilde{L}_t = \bigcup_{\alpha\in \mathcal{V}}L_\alpha$ 
in the sense of current.
\label{gluing}
\end{thm}

To obtain Theorem \ref{main1},
we apply Theorem \ref{gluing} to the case that 
$M$ is a toric \hK manifold 
and $L_\alpha$ is a \hL submanifold with respect to $\cos(\alpha\pi/n) I_2 + \sin(\alpha\pi/n) I_3$.
Accordingly, the proof is reduced to 
looking for toric \hK manifolds $M$ and their \hL submanifolds $L_1,\cdots,L_{2n}$
satisfying the assumption of Theorem \ref{gluing}. 
In particular, to find $L_\alpha$'s so that 
$\mathcal{E}$ is covered by cycles  
is not so easy. 
The author cannot develop the systematic way to find 
such examples in toric \hK manifolds, 
however, we can raise some concrete examples in this article.

In toric \hK manifolds, many \hL submanifolds 
are obtained as the inverse image of 
some special polytopes by the \hK moment maps, 
where the polytopes are naturally given by the 
hyperplane arrangements which determine 
the toric \hK manifolds. 
We can compute the type at the intersection point 
of two \hL submanifolds, if the intersection point 
is the fixed point of the torus action.
Finally, we can find examples of 
toric \hK manifolds and such polytopes, 
which satisfy the assumption Theorem \ref{gluing}.

Next we have to show that these examples of 
\sL submanifolds never become \hL submanifolds.
Since $\tilde{L}_t$ is contained in the homology 
class $\sum_{\alpha}(-1)^\alpha [L_{\alpha}]$, 
we obtain the volume of $\tilde{L}_t$ by integrating 
the real part of the holomorphic volume form 
over $\sum_{\alpha}(-1)^\alpha [L_{\alpha}]$.
On the other hand, if $\tilde{L}_t$ is \hL submanifold 
with respect to some $aI_1+bI_2+cI_3$, 
then the volume can be also computed by integrating 
$(a\omega_1+b\omega_2+c\omega_3)^n$ over 
$\sum_{\alpha}(-1)^\alpha [L_{\alpha}]$, 
since $a\omega_1+b\omega_2+c\omega_3$ should be the 
K\"ahler form on $\tilde{L}_t$.
These two values of the volume do not coincide, 
we have a contradiction.
At the same time, we have another simpler proof 
if the first betti number $\tilde{L}_t$ is odd, 
since any \hL submanifolds become K\"ahler manifolds 
which always have even first betti number.
The example constructed in Theorem \ref{main1} 
satisfies $b_1 = 1$, hence we can use this proof.
However, we have other examples 
in Section \ref{sec6} whose first betti number may be even.

This article is organized as follows.
First of all we define $\sigma$-\hL submanifolds in Section \ref{sec2} 
and review the constructions of them in 
toric \hK manifolds in Section \ref{sec3}.
Next we review the definition of the type at the intersection point 
of two \sL submanifolds, and then compute them 
in the case of toric \hK manifolds in Section \ref{sec4}.
In Section \ref{sec5}, we prove Theorem \ref{gluing} by 
using Theorem 9.7 of \cite{joyce2003special}.
In Section \ref{sec6}, we find toric \hK manifolds 
and their \hL submanifolds which satisfy the assumption of 
Theorem \ref{gluing}, and obtain compact \sL submanifolds 
embedded in some toric \hK manifolds.
In Section \ref{sec7}, we show the examples obtained in 
Section \ref{sec6} never become $\sigma$-\hL submanifolds 
for any $\sigma\in S^2$.

{\bf Acknowledgment.} The author would like to express his gratitude 
to Professor Dominic Joyce for his advice on this article. 
The author is also grateful to Dr. Yohsuke Imagi for useful discussion
and his advice.

\section{Holomorphic Lagrangian submanifolds}\label{sec2}
\begin{definition}
{\rm A Riemannian manifold $(M,g)$ equipped with integrable complex structures $(I_1,I_2,I_3)$ is a}
\hK manifold {\rm if each $I_{\alpha}$ is orthogonal with respect to $g$, they satisfy the quaternionic 
relation $I_1I_2I_3 = -1$ and fundamental $2$-forms $\omega_{\alpha}:=g(I_{\alpha}\cdot,\cdot)$ are closed.}
\end{definition}
We put $\omega=(\omega_1,\omega_2,\omega_3)$ and call it the \hK structure. 
For each 
\begin{eqnarray}
\sigma = (\sigma_1,\sigma_2,\sigma_3) \in S^2=\{ (a,b,c)\in \R^3;\ a^2+b^2+c^2=1 \},\nonumber
\end{eqnarray}
we have another K\"ahler structure 
\begin{eqnarray}
(M, I^\sigma, \omega^\sigma) := (M, \sum_{i=1}^3 I_i\omega_i, \sum_{i=1}^3\sigma_i\omega_i). \nonumber
\end{eqnarray}
Take $\sigma',\sigma''\in S^2$ so that $(\sigma,\sigma',\sigma'')$ forms an orthonormal 
basis in $\R^3$. 
Suppose it has the positive orientation, that is, 
\begin{eqnarray}
\sigma \wedge \sigma' \wedge \sigma'' = 
(1,0,0) \wedge (0,1,0) \wedge (0,0,1)\nonumber
\end{eqnarray}
holds.
Then we have another \hK structure $(\omega^\sigma, \omega^{\sigma'}, \omega^{\sigma''})$ 
which is called the \hK rotation of $\omega$.
\begin{definition}
{\rm Let $(M,g,I_1,I_2,I_3)$ be a \hK manifold of real dimension $4n$, and $L\subset M$ be a 
$2n$-dimensional submanifold. 
Fix $\sigma\in S^2$ arbitrarily. 
Then $L$ is a $\sigma$}-\hL submanifold {\rm if $\omega^{\sigma'}|_{L} = \omega^{\sigma''}|_{L} = 0$.}
\end{definition}
It is easy to see that the above definition is not depend on the choice of $\sigma',\sigma''$.

Any \hK manifolds can be regarded as Calabi-Yau manifolds by considering the pair of a K\"ahler manifold
$(M,I_1,\omega_1)$ and a holomorphic volume form $(\omega_2+\sqrt{-1}\omega_3)^n \in H^0(M,K_M)$, where 
$K_M$ is the canonical line bundle of the complex manifold $(M,I_1)$.
Therefore, we can consider the notion of \sL submanifolds in $M$ as follows.
\begin{definition}
{\rm Let $(M,g,I_1,I_2,I_3)$ be a \hK manifold of real dimension $4n$, and $L\subset M$ be a 
$2n$-dimensional submanifold. 
Then $L$ is a} \sL submanifold {\rm if $\omega_1|_{L} = {\rm Im}(\omega_2+\sqrt{-1}\omega_3)^n|_L = 0$.}
\end{definition}

\section{Toric \hK manifolds}\label{torichk}\label{sec3}
\subsection{Construction}\label{const}
In this subsection we review the construction of toric \hK manifolds shortly.
Let $u_{\mathbb{Z}} : \mathbb{Z}^d \to \mathbb{Z}^n$ be a surjective $\mathbb{Z}$ linear map which induces a 
homomorphisms between tori and their Lie algebras, denoted by 
$\hat{u}: T^d \to T^n$ and $u : \mathbf{t}^d \to \mathbf{t}^n$, respectively.
We put $K:={\rm Ker}\ \hat{u}\in T^d$ and $\mathbf{k}:= {\rm Ker}\ u\in \mathbf{t}^d$, 
where $\mathbf{k}:$ is the Lie algebra of the subtorus $K$. 
The adjoint map of $u$ is denoted by $u^* : (\mathbf{t}^n)^*\to (\mathbf{t}^d)^*:$ and it induces 
$u^*:V\otimes (\mathbf{t}^n)^*\to V\otimes (\mathbf{t}^d)^*$ naturally for any vector space $V$, 
which is also denoted by the same symbol.

Next we consider the action of $T^d$ on the quaternionic vector space $\quater^d$ 
given by $(x_1,\cdots,x_d)\cdot (g_1,\cdots, g_d) := (x_1g_1,\cdots,x_dg_d)$
for $x_k\in \quater$ and $g_k\in S^1$.
Then this action preserves the standard \hK structure on $\quater^d$, and the \hK moment map 
$\mu_d:\quater^d\to \ImH \otimes (t^d)^*$ is given by
$\mu_d(x_1,\cdots,x_d) = (x_1i\overline{x_1},\cdots,x_di\overline{x_d})$. 
Here, $\ImH \cong \R^3$ is the pure imaginary part of $\quater$.

Let $\hat{\iota} : K\to T^d$ and $\iota : \mathbf{k}\to \mathbf{t}^d$ be the inclusion maps and put 
$\mu_K:=\iota^*\circ\mu_d : \quater^d\to \ImH\otimes\mathbf{k}^*$ be the \hK moment map 
with respect to $K$-action on $\quater^d$. For each $\lambda \in \ImH\otimes \mathbf{k}^*$, 
we obtain the \hK quotient $X(u,\lambda):= \mu_K^{-1}(\iota^*(\lambda))/K$ for every 
$\lambda=(\lambda_1,\cdots,\lambda_d)\in \ImH\otimes (\mathbf{t}^d)^*$, called toric \hK varieties.
The complex structures on $X(u,\lambda)$ are denoted by 
$I_{\lambda,1}, I_{\lambda,2},I_{\lambda,3}$, and the 
corresponding K\"ahler forms are denoted by 
$\omega_{\lambda}=(\omega_{\lambda,1}, \omega_{\lambda,2},\omega_{\lambda,3})$.

Although $X(u,\lambda)$ is not necessarily to be a smooth manifold, 
the equivalent condition for the smoothness was obtained 
by Bielawski-Dancer in \cite{bielawski2000geometry}. 
Let $e_1,\cdots,e_d\in \R^d$ be the standard basis and $u_k:=u(e_k) \in \mathbf{t}^n$.
Put 
\begin{eqnarray}
H_k = H_k(\lambda) := \{ y\in \ImH\otimes (\mathbf{t}^n)^* ;\ \langle y, u_k \rangle + \lambda_k = 0 \},\nonumber
\end{eqnarray}
where 
\begin{eqnarray}
\langle y, u_k \rangle = (\langle y_1, u_k \rangle,\ \langle y_2, u_k \rangle,\ \langle y_3, u_k \rangle) \in \R^3 = \ImH \nonumber
\end{eqnarray}
for $y = (y_1,y_2,y_3)$.
\begin{thm}[\cite{bielawski2000geometry}]
The \hK quotient $X(u,\lambda)$ is a smooth manifold if and only if 
both of the following conditions $(*1)(*2)$ are satisfied.
$(*1)$ For any $\tau \subset \{1,2,\cdots, d \}$ with $\#\tau = n + 1$, 
the intersection $\bigcap_{k\in \tau} H_k$ is empty.
$(*2)$ For every $\tau \subset \{1,2,\cdots, d \}$ with $\#\tau = n$, 
the intersection $\bigcap_{k\in \tau} H_k$ is nonempty if and only if 
$\{u_k;\ k\in\tau \}$ is a $\Z$-basis of $\Z^n$.
\label{smooth}
\end{thm}

The $T^d$ action on $\quater^d$ induces a $T^n=T^d/K$ action on $X(u,\lambda)$ preserving 
the \hK structure of $X(u,\lambda)$, and the \hK moment map 
$\mu_{\lambda} = (\mu_{\lambda,1},\ \mu_{\lambda,2},\ \mu_{\lambda,3}): X(u,\lambda)\to \ImH\otimes (\mathbf{t}^n)^*$ is defined by 
\begin{eqnarray}
u^*(\mu_{\lambda}([x])):= \mu_d(x) - \lambda,\nonumber
\end{eqnarray}
where $[x]\in X(u,\lambda)$ is the equivalence class represented by $x\in \mu_K^{-1}(\iota^*(\lambda))$.

Let $\sigma\in S^2$. 
A $T^n$-invariant submanifold $L\subset X(u,\lambda)$ becomes a $\sigma$-\hL submanifold 
if $\mu_\lambda(L)$ is contained in $q+\sigma\otimes (\mathbf{t}^n)^*$ for some $q\in \ImH\otimes (\mathbf{t}^n)^*$.

\subsection{Local model of the neighborhood of a fixed point}\label{localmodel}
Let $X=X(u,\lambda)$ be a smooth toric \hK manifold of real dimension $4n$, 
$\omega=\omega_{\lambda}$ and $\mu=\mu_{\lambda}$. 
Denote by $X^*$ the maximal subset of $X$ on whom 
$T^n$ acts freely.
Let $p\in X$ be a fixed point of the $T^n$-action.
Then we can see that 
\begin{eqnarray}
H_{k_1}\cap H_{k_2}\cap \cdots\cap H_{k_n} = \{ \mu(p) \}\nonumber
\end{eqnarray}
for some $k_1,\cdots, k_n$, and we may suppose 
$k_i=i$ without loss of generality. 

By the result of \cite{pedersen1988hyper}, the \hK metric on $X^*$ can be described 
using $\mu$ and $T^n$-connection on $X^*$ and some functions defined on 
$\mu(X^*)$. 
Using their result, we can see that $\omega$ can be decomposed into two parts as
\begin{eqnarray}
\omega = \omega_{\quater^n} + \mu^*\eta \nonumber
\end{eqnarray}
on $U$, where $U$ is a $T^n$-invariant neighborhood of 
$p$, $\omega_{\quater^n}$ has the same form with 
the standard \hK structure on $\quater^n$, 
and $\eta\in \Omega^2(\mu(U))$.
Hence we have the followings.
\begin{prop}
Let $(X,\omega,\mu)$ and $p$ be as above. 
There is a $T^n$-invariant neighborhood $U\subset X$ of $p$, 
a $T^n$-equivariant diffeomorphism $F:U \to B_{\quater^n}(\varepsilon)$ and 
$\eta\in \Omega^2(\mu(U))$ which satisfy
\begin{eqnarray}
\omega|_U &=& F^*\omega_{\quater^n}|_U+ \mu^*\eta, \nonumber\\
\mu_n|_{B_{\quater^n}(\varepsilon)} \circ F &=& \mu|_U - \mu(p) \nonumber
\end{eqnarray}
where $B_{\quater^n}(\varepsilon)=\{ x\in \quater^n;\ \| x\| < \varepsilon\}$, 
and $\omega_{\quater^n}$ is the standard \hK structure on $\quater^n$.
\label{3.2}
\end{prop}

\section{Characterizing angles}\label{sec4}
\subsection{Calabi-Yau case}
For the desingularization of special Lagrangian immersions which intersect transversely on a point, 
one should consider the characterizing angles, introduced by Lawlor \cite{lawlor1989angle}.

Let $(M,J,\omega)$ be a K\"ahler manifold, where $J$ is a complex structure,
$\omega$ is a K\"ahler form.
Suppose that there is a Lagrangian immersion $\iota : L \to M$, where 
$\iota$ is embedding on $L\backslash \{p_+,p_-\}$ and $\iota(L)$ intersects at 
$\iota(p_+) = \iota(p_-) = p \in M$ transversely.
We suppose $L$ is not necessarily to be connected, 
and the orientation of $L$ is fixed.
\begin{thm}[\cite{joyce2003special}\cite{joyce2004special}]
Let $(J_0,\omega_0)$ be the standard K\"ahler structure
on $\C^m$.
There exists a linear map $v:T_pM\to \C^m$ satisfying the following conditions;
(i) $v$ is a $\C$-linear isomorphism preserving the K\"ahler forms, 
(ii) there is $\varphi =(\varphi_1, \cdots, \varphi_m)\in \R^m$ 
which satisfies 
$0<\varphi_1\le \cdots\le \varphi_m<\pi$ and
\begin{eqnarray}
v\circ\iota_* (T_{x_+} L) &=& \R^m = \{ (t_1,\cdots, t_m)\in \C^m;\  t_i\in\R \},\nonumber\\
v\circ\iota_* (T_{x_-} L) &=& \R^m_{\varphi}=\{ (t_1 e^{\sqrt{-1}\varphi_1},\cdots, t_m e^{\sqrt{-1}\varphi_m})\in \C^m;\ t_i\in\R \}.\nonumber
\end{eqnarray}
(iii) $v$ maps the orientation of $\iota_* (T_{x_+} L)$ to 
the standard orientation of $\R^m$.
Moreover, $\varphi_1, \cdots, \varphi_m$ and the induced orientation 
of $\R^m_{\varphi}$ by $v,\iota_* (T_{x_-} L)$ 
do not depend on the choice of $v$.
\label{joyce}
\end{thm}
\begin{proof}
Choose a $\C$-linear map $v_0:T_pM\to \C^m$ which preserves the K\"ahler metrics 
and the orientations of $\iota_* (T_{x_+} L)$ and $\R^m$.
Then $V_{\pm}:=v_0\circ\iota_* (T_{x_{\pm}} L)$ are Lagrangian subspaces of $\C^m$. 
It is well-known that any Lagrangian subspaces in $\C^m$ are written as $g\cdot\R^m$ for some 
$g\in U(m)$, where $\R^m\subset \C^m$ is the standard real form of $\C^m$.
Therefore there are $g_{\pm}\in U(m)$ so that $g_+\cdot V_+ = g_-\cdot V_- = \R^m$.
We may chose $g_+$ so that it preserves the orientations of $V_+$ and $\R^m$.
Once we fined such $g_\pm$, we can replace them by $h_\pm g_\pm$ 
for any $h_+\in SO(m)$ and $h_-\in O(m)$ respectively. 
Now, let $v:= h_+ g_+ v_0$. 
Then we have $v\circ\iota_* (T_{x_+} L) = \R^m$ and 
\begin{eqnarray}
v\circ\iota_* (T_{x_-} L) &=& h_+ g_+ V_- = (h_+ g_+g_-^{-1}h_-^{-1})h_-g_- V_- \nonumber\\
&=& (h_+ g_+g_-^{-1}h_-^{-1}) \R^m. \nonumber
\end{eqnarray}
Accordingly, it suffices to show that $h_\pm \in O(m)$ can be chosen so that
$h_+ g_+g_-^{-1}h_-^{-1}$ is a diagonal matrix.
Put $P= g_+g_-^{-1} \in SU(m)$. 
Since ${}^tPP $ is a unitary and symmetric matrix, it can be diagonalized by some $Q\in O(m)$, that is, 
${}^tQ{}^tPP Q = {\rm diag}(e^{\sqrt{-1}\theta_1},\cdots,e^{\sqrt{-1}\theta_m})$ holds for some 
$0\le \theta_1\le \cdots \le\theta_m <2\pi$. 
Note that $Q$ can be chosen so that either ${\rm det}(Q) = 1$ or ${\rm det}(Q) = -1$. 
If we put $R:= P Q {\rm diag}(e^{-\sqrt{-1}\theta_1/2},\cdots,e^{-\sqrt{-1}\theta_m/2}) \in U(m)$, then 
${}^tR=R^*=R^{-1}$ holds, hence $R$ is contained in $O(m)$.
We determine the value of ${\rm det}(Q)$ so that ${\rm det}(R) = 1$.
Hence the assertion follows by putting $h_+ = R^{-1}$, $h_- = Q^{-1}$ and  
$\varphi_i = \theta_i /2$.
Here, $\varphi_i$ never be $0$ since $V_+$ and $V_-$ intersect transversely.

Next we show the uniqueness of $\varphi_1\le \cdots \le\varphi_m$ and the 
orientation of $\R^m_{\varphi}$.
Assume that we have other $\hat{v}:T_pM\to \C^m$ satisfying 
$(i)(ii)(iii)$, and suppose 
\begin{eqnarray}
\hat{v}\circ\iota_* (T_{x_-} L) = \{ (t_1 e^{\sqrt{-1}\hat{\varphi}_1},\cdots, t_m e^{\sqrt{-1}\hat{\varphi}_m})\in \C^m;\ t_i\in\R \}\nonumber
\end{eqnarray}
holds for some $0<\hat{\varphi}_1 \le \cdots \le \hat{\varphi}_m < \pi$.
If we put $\hat{g}:=\hat{v}v^{-1}$, then $\hat{g}$ is in $U(m)$ and preserves 
the subspace $\R^m\subset\C^m$, hence $\hat{g} \in O(m)$.
Moreover $\hat{g}\in SO(m)$ holds since $v$ and $\hat{v}$ satisfy 
$(iii)$.
Moreover $\hat{g}(\R^m_{\varphi}) = \R^m_{\hat{\varphi}}$ also holds,
consequently we can see that
\begin{eqnarray}
G:= {\rm diag}(e^{-\sqrt{-1}\hat{\varphi}_1},\cdots,e^{-\sqrt{-1}\hat{\varphi}_m})\cdot \hat{g}\cdot {\rm diag}(e^{\sqrt{-1}\varphi_1},\cdots,e^{\sqrt{-1}\varphi_m})\nonumber
\end{eqnarray}
is a real matrix, hence we can deduce that 
\begin{eqnarray}
\hat{g}\cdot {\rm diag}(e^{2\sqrt{-1}\varphi_1},\cdots,e^{2\sqrt{-1}\varphi_m})\cdot \hat{g}^{-1}
= {\rm diag}(e^{2\sqrt{-1}\hat{\varphi}_1},\cdots,e^{2\sqrt{-1}\hat{\varphi}_m})
\nonumber
\end{eqnarray}
by $\overline{G}=G$.
Thus we obtain $e^{2\sqrt{-1}\varphi_i}=e^{2\sqrt{-1}\hat{\varphi}_i}$ 
for all $i=1,\cdots,m$, 
which implies $\varphi_i=\hat{\varphi}_i$, since 
$\hat{\varphi}_i$ are taken from $(0,\pi)$. 
Now we have two orientations on $\R^m_{\varphi}$ induced from 
$v$ and $\hat{v}$ respectively. 
These coincide since 
${\rm det}(G) = {\rm det}(\hat{g}) = 1$.
\end{proof}

Here, $\varphi=(\varphi_1, \cdots, \varphi_m)$ is called the characterizing angles between $(L,p_+)$ and $(L,p_-)$.
Under the above situation, 
assume that there is a holomorphic volume form $\Omega$ on $M$ satisfying 
$\omega^m/m!= (-1)^{m(m-1)/2}(\sqrt{-1}/2)^m\Omega\wedge\overline{\Omega}$, where 
$m$ is the complex dimension of $M$.
Let $\Omega_0:=dz_1\wedge \cdots\wedge dz_m$ be the standard holomorphic 
volume form on $\C^m$, 
and assume that $\iota : L\to M$ is a \sL immersion.
Then there exists $v:T_pM\to \C^m$ satisfying Theorem \ref{joyce}. 
By the condition $(ii)$, we can see that $v^*\Omega_0 = \Omega_p$.

Since both of $\iota_* (T_{p_\pm} L)$ are special Lagrangian subspaces, 
there is a positive integer $k = 1,2,\cdots m-1$ and $\varphi_1+ \cdots + \varphi_m =k\pi$ holds. 
Then the intersection point $p\in M $ is said to be of type $k$.
Note that the type depends on the order of $p_+,p_-$. 
If we take the opposite order, the characterizing angles become 
$\pi -\varphi_m,\cdots,\pi -\varphi_1$ and the type becomes $m-k$.

\subsection{\HK case}
An irreducible decomposition of $T^n$-action on $\quater^n$ is given by
\begin{eqnarray}
\quater^n = \bigoplus_{i=1}^n Z_i \oplus \bigoplus_{i=1}^n W_i, \nonumber
\end{eqnarray}
where $Z_i$ and $W_i$ are complex $1$-dimensional representation of $T^n$
defined by 
\begin{eqnarray}
(g_1,\cdots, g_n)z_i:= g_i z_i,\quad (g_1,\cdots, g_n)w_i:= g_i^{-1} w_i\nonumber
\end{eqnarray}
for $(g_1,\cdots, g_n)\in T^n$ and $z_i\in Z_i,\ w_i\in W_i$.
Note that $Z_i$ and $W_i$ are not isomorphic as $\C$-representations, 
but the complex conjugate restricted to $Z_i$ gives an 
isomorphism of $\R$-representations $Z_i\to W_i$.
For $(\alpha,\beta)\in S^3 \subset \C^2$, put 
$h(\alpha,\beta) := (|\alpha|^2 - |\beta|^2, 2{\rm Im}(\alpha \beta), -2{\rm Re}(\alpha \beta)) \in \R^3$.
Then $h:S^3\to S^2$ is the Hopf fibration and $S^1$ action is given by 
$e^{\sqrt{-1}t}\cdot(\alpha,\beta) = (e^{\sqrt{-1}t}\alpha,e^{-\sqrt{-1}t}\beta)$.

Now we put 
\begin{eqnarray}
V_i(y) := \{(\alpha z_i, \beta \overline{z_i})\in Z_i\oplus W_i;\ z_i\in Z_i\}\nonumber
\end{eqnarray}
for $y\in S^2$, where $(\alpha,\beta)\in S^3$ is taken to be $h(\alpha,\beta) = y$.
Then $V_i(y)$ does not depend on the choice of $(\alpha,\beta)$,
and $V_i(y)$ is an sub $\R$-representation of $Z_i\oplus W_i$.
Conversely, any nontrivial sub $\R$-representation of $Z_i\oplus W_i$ is obtained in this way.
Note that $V_i(y) = V_i(y')$ holds if and only if $y=y'$.

\begin{prop}
Let $V\subset \quater^n$ be a $\sigma$-\hL subspace 
which is closed under the $T^n$ action. 
Then we have 
\begin{eqnarray}
V = \bigoplus_{i=1}^n V_i(\varepsilon_i \sigma) \nonumber
\end{eqnarray}
for some $\varepsilon_i = \pm 1$,
and its \hK moment image is given by
\begin{eqnarray}
\mu_n(V) = \{ \sigma\otimes (x_1,\cdots, x_n)\in \ImH \otimes (\mathbf{t}^n)^*;\ \varepsilon_i x_i\ge 0 \}. \nonumber
\end{eqnarray}
\label{model}
\end{prop}

\begin{proof}
Take $\sigma',\sigma''\in S^2$ so that $(\sigma,\sigma',\sigma'')$ is the orthonormal basis 
with positive orientation of $\R^3$.
Let $I_1,I_2,I_3$ be the standard basis of the pure imaginary part of $\quater$, and 
$I^\sigma,I^{\sigma'},I^{\sigma''}$ be its \hK rotation.
Let $V\subset \quater^n$ be a $\sigma$-holomorphic Lagrangian subspace 
which is closed under the $T^n$ action.
Since $V$ is Lagrangian with respect to $\omega^{\sigma'}$, we have an orthogonal decomposition
$\quater^n = V \oplus I^{\sigma'} V$. 
Then $V$ and $I^{\sigma'} V$ are isomorphic as real representations of $T^n$, 
therefore $V$ should be isomorphic to $\oplus_{i=1}^n Z_i$ as real representations of $T^n$, 
and can be written as $V = \oplus_{i=1}^nV_i(y_i)$ for some $y_i\in S^2$ 
by Schur's Lemma.
Here, every $y_i$ is determined uniquely.
Next we calculate the restriction of $\omega$ to $V$.
Let $(z_1,\cdots, z_n,w_1,\cdots, w_n)$ be the holomorphic coordinate on $\quater^n$, 
where $z_i\in Z_i \cong \C$ and $w_i \in W_i \cong \C$ then $\omega$ can be written as 
\begin{eqnarray}
\omega_1 &=& \frac{\sqrt{-1}}{2}\sum_{i=1}^n(dz_i\wedge d\overline{z_i} + dw_i\wedge d\overline{w_i}),\nonumber\\
\omega_2+\sqrt{-1}\omega_3 &=& \sum_{i=1}^n dz_i\wedge dw_i.\nonumber
\end{eqnarray}
Take $(\alpha_i,\beta_i)\in S^3$ such that $h(\alpha_i,\beta_i) = y_i$. 
Then $P,Q\in V=\oplus_{i=1}^nV_i(y_i)$ can be written as 
\begin{eqnarray}
P &=& (\alpha_1p_1,\cdots, \alpha_n p_n, \beta_1\overline{p_1},\cdots \beta_n\overline{p_n}),\nonumber\\
Q &=& (\alpha_1q_1,\cdots, \alpha_n q_n, \beta_1\overline{q_1},\cdots \beta_n\overline{q_n})\nonumber
\end{eqnarray}
for some $p_i\in Z_i$ and $q_i\in W_i$.
Then we obtain 
\begin{eqnarray}
\omega_1(P,Q) &=& \sum_{i=1}^n(|\alpha_i|^2 - |\beta_i|^2 ){\rm Im}(\overline{p_i}q_i),\nonumber\\
(\omega_2+\sqrt{-1}\omega_3)(P,Q) &=& -2\sqrt{-1}\sum_{i=1}^n \alpha_i\beta_i{\rm Im}(\overline{p_i}q_i).\nonumber
\end{eqnarray}
Hence $\sigma$ Lagrangian condition for $V$ is equivalent to that the vector 
\[ 
\left (
\begin{array}{c}
 \sum_{i=1}^n(|\alpha_i|^2 - |\beta_i|^2 ){\rm Im}(\overline{p_i}q_i) \\
2\sum_{i=1}^n {\rm Im}(\alpha_i\beta_i){\rm Im}(\overline{p_i}q_i) \\
-2\sum_{i=1}^n {\rm Re}(\alpha_i\beta_i){\rm Im}(\overline{p_i}q_i)
\end{array}
\right )\in \R^3
\]
is orthogonal to $\sigma',\sigma''\in \R^3$ for any $p_i,q_i\in\C$.
Thus every 
\[ y_i = 
\left (
\begin{array}{c}
 |\alpha_i|^2 - |\beta_i|^2  \\
2 {\rm Im}(\alpha_i\beta_i) \\
-2 {\rm Re}(\alpha_i\beta_i)
\end{array}
\right )
\]
is equal to $\pm \sigma$ because $\{\sigma,\sigma',\sigma''\}$ is an orthonormal basis.
\end{proof}

\begin{prop}
Let $(X,\omega,\mu)$ and $p$ be as in Proposition \ref{3.2}. 
Let $L\subset X$ be a $\sigma$-\hL submanifold containing $p$, and 
assume that there exists a sufficiently small $r >0$, $\varepsilon_i,\varepsilon'_i = \pm 1$, and
\begin{eqnarray}
(\mu (L)-\mu(p))\cap B(r) &=& \sigma \otimes \{ x \in (\mathbf{t}^n)^*;\ 
\| x\| < r, \varepsilon'_i x_i \ge 0\},\nonumber\\
\mu_n (V) &=& \sigma \otimes \{ x \in (\mathbf{t}^n)^*;\ 
\varepsilon_i x_i \ge 0\}\nonumber
\end{eqnarray}
holds, where $V= dF_p(T_p L)$ and $B(r) = \{ y\in\ImH \otimes (\mathbf{t}^n)^*; \| y\| <r\}$.
Then $\varepsilon_i = \varepsilon'_i$ holds for every $i=1,\cdots, n$.
\label{apex}
\end{prop}
\begin{proof}
By the first equation of Proposition \ref{3.2}, we have 
$\mu (L)-\mu(p) = \mu_n(F(L))$.
Since $T_0F(L) = dF_p(T_pL)$, $\mu_n (V) = \mu_n(T_0F(L))$ holds.
Now we have open neighborhoods $U_0\subset F(L)$ of $0$, 
$U_1\subset T_0F(L)$ of $0$, and a diffeomorphism $f:U_0\to U_1$ such that 
$f(0)= 0$ and $df_0= {\rm id}$.
Next we take a smooth map $\gamma:(-1,1) \to F(L)$ which satisfies 
$\gamma(0) = F(p) = 0$, $\varepsilon'_i\mu_n^i(\gamma(t)) >0$ for $t\neq 0$. 
Here $\mu_n^i$ is the $i$-th component of $\mu_n$.
Since $\| f(x)-x\| = \mathcal{O}(\| x\|^2)$ and 
$\mu_n(x + \delta x) = \mu_n(x) + \mathcal{O}(\| x\| \| \delta x\|)$ holds, 
we have 
\begin{eqnarray}
\mu_n(f\circ\gamma(t)) &=& \mu_n(\gamma(t) + \mathcal{O}(\| \gamma(t)\|^2))\nonumber\\
&=& \mu_n(\gamma(t)) + \mathcal{O}(\| \gamma(t)\|^3)\nonumber
\end{eqnarray}
If we take $t$ sufficiently close to $0$, then $\| \gamma(t)\|$ is sufficiently small 
but $\varepsilon'_i\mu_n^i(\gamma(t)) >0$, hence 
$\varepsilon'_i\mu_n^i(f\circ\gamma(t))$ should be positive 
for small $t$ since $\mu_n$ is a quadratic polynomial. 
Since $\mu_n(f\circ\gamma(t)) \in \mu_n (V)$, $\varepsilon_i = \varepsilon'_i$ must holds. 
We have taken $i$ arbitrarily, $\varepsilon_i = \varepsilon'_i$ holds for every $i=1,\cdots,n$.
\end{proof}

Let 
\begin{eqnarray}
\sigma(\theta)=(0,\cos\theta,\sin\theta) \in S^2.\nonumber
\end{eqnarray}
Then every $\sigma(\theta)$-holomorphic Lagrangian submanifold is 
special Lagrangian, if $n\theta \in \pi \mathbb{Z}$.

\begin{prop}
Let $n\theta_{\pm} \in \pi\mathbb{Z}$ and $V_{\pm}$ be 
$T^n$-invariant $\sigma(\theta_{\pm})$-holomorphic Lagrangian subspaces 
of $\quater^n$ given by 
\begin{eqnarray}
V_+ := \bigoplus_{i=1}^n V_i( \sigma(\theta_+) ),\quad
V_- := \bigoplus_{i=1}^n V_i( \sigma(\theta_-) ).\nonumber
\end{eqnarray}
Then the characterizing angles between $V_+$ and $V_-$ are given by 
$(\theta_{-} - \theta_{+})/2$ with multiplicity $2n$.
\label{angle}
\end{prop}

\begin{proof}
Since $h(\sqrt{-1}/\sqrt{2},e^{\sqrt{-1} \theta_{\pm}} / \sqrt{2}) = \sigma(\theta_{\pm})$, 
we have 
\begin{eqnarray}
V_{\pm}=\{(\frac{\sqrt{-1}}{\sqrt{2}}z_1, \frac{e^{\sqrt{-1} \theta_{\pm}}}{\sqrt{2}}\overline{z_1}, \cdots, 
\frac{\sqrt{-1}}{\sqrt{2}}z_n, \frac{e^{\sqrt{-1} \theta_{\pm}}}{\sqrt{2}}\overline{z_n}) \in \quater^n;\ 
z_1,\cdots, z_n\in \C \}\nonumber
\end{eqnarray}
respectively.
Put 
\[ A(\theta) := \frac{1}{\sqrt{2}}
\left (
\begin{array}{ccc}
-\sqrt{-1} & e^{-\sqrt{-1} \theta}   \\
-1 & \sqrt{-1}e^{-\sqrt{-1} \theta}
\end{array}
\right ),
\]
and 
\[ g_+ := 
\left (
\begin{array}{ccc}
A(\theta_{+}) &  & \rm{O}  \\
 & \ddots & \\
\rm{O} & & A(\theta_{+})
\end{array}
\right ),\quad
g_- := 
\left (
\begin{array}{ccc}
A(\theta_{-}) &  & \rm{O}  \\
 & \ddots & \\
\rm{O} & & A(\theta_{-})
\end{array}
\right ).
\]
Since $g_+V_+ = g_-V_- = \R^{2n}$ holds, 
then the characterizing angles are the argument of the square root of the 
eigenvalues of  ${}^tPP$, where $P = g_+g_-^{-1}$, 
by the proof of Theorem \ref{joyce}.
Since 
\begin{eqnarray}
{}^t(A(\theta_{+})A(\theta_{-})^{-1})A(\theta_{+})A(\theta_{-})^{-1} 
= e^{\sqrt{-1}(\theta_{-} - \theta_{+})}{\rm Id},\nonumber
\end{eqnarray}
the characterizing angles turn out to be $(\theta_{-} - \theta_{+})/2$ with multiplicity $2n$.
\end{proof}

Now we consider the case that 
\begin{eqnarray}
(M,J,\omega,\Omega)=(X(u,\lambda),I_1,(\omega_{\lambda,2} + \sqrt{-1} \omega_{\lambda,3})^n)\nonumber
\end{eqnarray}
and $L=L_+\sqcup L_- $, where $L_{\pm}$ is embedded as $\sigma(\theta_{\pm})$-holomorphic Lagrangian submanifolds 
respectively, for some $\theta_{\pm}\in\R$. 
Denote by $\iota: L\to X(u,\lambda)$ the immersion.
Assume that the image of $L$ is a $T^n$ invariant subset of $X(u,\lambda)$, and $p$ is the fixed point of the torus action.
In this subsection, we see the characterizing angles between $(L,p_+)$ and $(L,p_-)$ in this situation.

Take $F:U\to B_{\quater^n}(\varepsilon)$ as in Proposition \ref{3.2}. 
Then $dF_p:T_pX(u,\lambda) \to \quater^n$ is $T^n$-equivariant 
and satisfies $dF_p^*(\omega_{\quater^n}|_0) = \omega|_p$ 
by the first equation in Proposition \ref{3.2} 
since $d\mu_p=0$. 
Here, a $T^n$-action on $T_pX(u,\lambda)$ is 
induced from the torus action on $X(u,\lambda)$ since $p$ is fixed by the action.
Then $V_{\pm}:=dF_p\circ\iota_*(T_{p_{\pm}}L)$ is a 
$\sigma_{\pm}$-holomorphic Lagrangian subspace of $\quater^n$, respectively.
Moreover, $V_{\pm}$ are closed under the $T^n$-action.

\begin{prop}
Under the above setting, assume that there is a sufficiently small $r>0$ 
and 
\begin{eqnarray}
(\mu (L_\pm)-\mu(p))\cap B(r) &=& \sigma_{\pm} \otimes \{ x \in (\mathbf{t}^n)^*;\ 
\| x\| < r, x_i \ge 0\}\nonumber\nonumber
\end{eqnarray}
holds respectively. 
Then the characterizing angles between $(L,p_+)$ and $(L,p_-)$ are given by 
$(\theta_- - \theta_+)/2$ with multiplicity $2n$.
\end{prop}
\begin{proof}
By combining Propositions \ref{model}\ref{apex}, we can see that 
\begin{eqnarray}
V_\pm = \bigoplus_{i=1}^n V_i( \sigma(\theta_\pm) )\nonumber
\end{eqnarray}
respectively. 
Thus we have the assertion by Proposition \ref{angle}.
\end{proof}

\section{Proof of Theorem \ref{gluing}}\label{sec5}
In this section we prove Theorem \ref{gluing}.
Although Theorem \ref{gluing} follows from Theorem 9.7 of 
\cite{joyce2003special} essentially, 
we need some additional argument about the quivers.
Let $Q = (\mathcal{V},\mathcal{E},s,t)$ be a quiver, 
that is, $\mathcal{V}$ consists of finite vertices, 
$\mathcal{E}$ consists of directed finite edges, 
and $s,t:\mathcal{E} \to \mathcal{V}$ are maps.
Here, $s(h)$ and $t(h)$ means the source and 
the target of $h\in\mathcal{E}$ respectively. 
The quiver is said to be connected if 
any two vertices are connected by some edges.
Given the quiver, we have operators 
\begin{eqnarray}
\partial &:& \R^\mathcal{E} \to \R^\mathcal{V}\nonumber\\
\partial^* &:& \R^\mathcal{V} \to \R^\mathcal{E}\nonumber
\end{eqnarray}
defined by
\begin{eqnarray}
\partial (\sum_{h\in \mathcal{E}}A_h\cdot h) 
&:=& \sum_{k\in \mathcal{V}}A_h\cdot (s(h) - t(h)),\nonumber\\
\partial^*(\sum_{k\in \mathcal{V}}x_k\cdot k) &:=& 
\sum_{h\in \mathcal{E}}(x_{s(h)} - x_{t(h)})\cdot h.\nonumber
\end{eqnarray}
Here, $\R^\mathcal{E}$ and $\R^\mathcal{V}$ are the 
free $\R$-modules generated by elements of 
$\mathcal{E}$ and $\mathcal{V}$ respectively.
Since $\partial^*$ is the adjoint of $\partial$, we have 
\begin{eqnarray}
\mathbf{h}_0(Q) - \mathbf{h}_1(Q) = \#\mathcal{V} - \#\mathcal{E},\label{euler}
\end{eqnarray}
where $\mathbf{h}_0(Q) = {\rm dim\ Ker} \partial^*$ 
and $\mathbf{h}_1(Q) = {\rm dim\ Ker} \partial$.
Note that $\mathbf{h}_0(Q)$ is equal to the number of the connected components of $Q$.

We need the following lemmas for the proof of Theorem \ref{gluing}.

\begin{lem}
Let $Q$ be as above.
The set $(\R_{>0})^\mathcal{E}\cap {\rm Ker}(\partial)$ is nonempty 
if and only if $\mathcal{E}$ is covered by cycles.
\label{quiver1}
\end{lem}
\begin{proof}
Suppose that $\mathcal{E} = \bigcup_{\alpha} S_k$ holds for some cycles 
$S_1,\cdots, S_N$. 
For a subset $S\subset \mathcal{E}$, define $\chi_S\in \R^\mathcal{E}$ by 
\[ (\chi_{S})_h :=
\left\{
\begin{array}{cc}
1 & (h\in S), \\
0 & (h\notin S).
\end{array}
\right.
\]
Then $\sum_{k=1}^N \chi_{S_k}$ is contained in $(\R_{>0})^\mathcal{E}\cap {\rm Ker}(\partial)$.

Conversely, assume that there exists 
$A = \sum_{h\in\mathcal{E}}A_h\cdot h\in (\R_{>0})^\mathcal{E}\cap {\rm Ker}(\partial)$,
and take $h_0\in \mathcal{E}$ arbitrarily. 
Since $\partial(A) = 0$, we have
\begin{eqnarray}
\sum_{h\in s^{-1}(t(h_0))}A_h = \sum_{h\in t^{-1}(t(h_0))}A_h \ge A_{h_0}>0.\nonumber
\end{eqnarray}
Hence $s^{-1}(t(h_0))$ is nonempty, we can take $h_1\in s^{-1}(t(h_0))$.
By repeating this procedure, we obtain $h_0,h_1,\cdots,h_l$ so that 
$t(h_k)=s(h_{k+1})$ for $k=0,\cdots,l-1$. 
Stop this procedure when $t(h_l) = s(h_k)$ holds for some $k=0,\cdots,l$. 
Since $\mathcal{V}$ is finite, this procedure always stops for some $l<+\infty$.
Then we have an nonempty cycle $S_0 = \{ h_k,h_{k+1},\cdots,h_l \}$. 
If $h_0$ is contained in $S_0$, then we have the assertion, hence suppose 
$h_0\notin S_0$
Put $A_0:=\min_{h\in S_0}A_h >0$, 
\begin{eqnarray}
P_0 &:=& \{ h\in\mathcal{E} ;\ A_h = A_0\},\nonumber\\
\mathcal{E}_1 &:=& \mathcal{E}\backslash P_0.\nonumber
\end{eqnarray}
Then we have a new quiver $((\mathcal{V},\mathcal{E}_1,s,t))$ and the 
boundary operator $\partial_1:\R^{\mathcal{E}_1} \to \R^\mathcal{V}$.
Now, put $A^{(1)}:=A- A_0\chi_{S_0} \in \R^{\mathcal{E}_1}$, where 
Then each component of $A^{(1)}$ is positive.
Moreover we can see that
\begin{eqnarray}
\partial_1 (A^{(1)}) &=& \sum_{h\in \mathcal{E} \backslash S_0}A_h (s(h) - t(h)) 
+ \sum_{h\in S_0\backslash P_0} (A_h - A_0)(s(h) - t(h)) \nonumber\\
&=& \sum_{h\in \mathcal{E}}A_h (s(h) - t(h)) - \sum_{h\in S_0}A_h (s(h) - t(h))\nonumber\\
&\quad & \quad
+ \sum_{h\in S_0} (A_h - A_0)(s(h) - t(h)) \nonumber\\
&=& \partial (A) - \sum_{h\in S_0}A_0(s(h) - t(h)) \nonumber\\
&=& - A_0\partial (\chi_{S_0}) = 0,\nonumber
\end{eqnarray}
thus $A^{(1)}$ is contained in $(\R_{>0})^{\mathcal{E}_1}\cap {\rm Ker}(\partial_1)$. 
Then we can apply the above procedure for $h_0\in \mathcal{E}_1$ and we can construct 
$S_k$ inductively.
Since $\mathcal{E}$ is finite and $\#\mathcal{E} > \#\mathcal{E}_1> \cdots $ holds, 
there is $k_0$ such that $h_0\in S_{k_0}$.
\end{proof}

\begin{lem}
Let $Q=(\mathcal{V},\mathcal{E},s,t)$ be as above. 
Then $Q'=(\mathcal{V},\mathcal{E}\backslash \{ h\},s,t)$ satisfies 
either $(\mathbf{h}_0(Q'),\mathbf{h}_1(Q')) = (\mathbf{h}_0(Q) + 1,\mathbf{h}_1(Q))$ 
or $(\mathbf{h}_0(Q'),\mathbf{h}_1(Q')) = (\mathbf{h}_0(Q),\mathbf{h}_1(Q) - 1)$ 
for any $h\in \mathcal{E}$.
\label{quiver2}
\end{lem}
\begin{proof}
Put 
\begin{eqnarray}
\mathcal{E}_1 &:=& \{ h\in \mathcal{E}; A_h 
= 0\ {\rm for\ any\ }A\in {\rm Ker}(\partial)\},\nonumber\\
\mathcal{E}_2 &:=& \mathcal{E}\backslash \mathcal{E}_1.\nonumber
\end{eqnarray}
First of all, we show that there exists 
$A\in (\R_{\neq 0})^{\mathcal{E}_2}\cap {\rm Ker}(\partial_2)$, 
where $\partial_2:\R^{\mathcal{E}_2} \to \R^\mathcal{V}$ is the restriction 
of $\partial$ to $\R^{\mathcal{E}_2} \subset \R^\mathcal{E}$.
By the definition of $\mathcal{E}_2$, we can easy to see that 
${\rm Ker}(\partial_2) = {\rm Ker}(\partial)$ holds, then we can take 
$A^h=\sum_{h'\in \mathcal{E}_2}A^h_{h'}\cdot h'\in {\rm Ker}(\partial_2)$ 
for every $h\in \mathcal{E}_2$ so that 
$A^h_h \neq 0$.
Since $\mathcal{E}_2$ is a finite set, we may write 
$\mathcal{E}_2 = \{ 1,\cdots, \#\mathcal{E}_2\}$.
Let $h_1$ be the minimum number so that $A^1_{h_1} = 0$.
Then put $B^2:= A^1 + a_1A^{h_1}$ for some $a_1\neq 0$. 
Chose $a_1$ sufficiently close to $0$ so that $B^2_h\neq 0$ for all 
$h\le h_1$. 
By defining $B^k$ inductively, finally we obtain 
$A = B^N\in (\R_{\neq 0})^\mathcal{E}_2\cap {\rm Ker}(\partial_2)$ 
for some $N$.

Since $\mathbf{h}_i(Q)$ is independent of the orientation of each edge in $\mathcal{E}$, 
we can replace $h\in \mathcal{E}_2$ by the edge with the 
opposite orientation if $A_h<0$. 
Consequently, we may suppose $A_h >0$ for any $h\in \mathcal{E}_2$ 
without loss of generality.
Hence $\mathcal{E}_2$ is covered by cycles by Lemma \ref{quiver1}.

Next we consider $Q'=(\mathcal{V},\mathcal{E}'=\mathcal{E}\backslash \{ h\},s,t)$.
Let $\partial':=\partial|_{\mathcal{E}'}$ and 
$(\partial')^*:\mathcal{V}\to\mathcal{E}'$ be the adjoint operator.
If $h\in \mathcal{E}_1$, then we can see ${\rm Ker}(\partial) = {\rm Ker}(\partial')$.
In this case $(\mathbf{h}_0(Q'),\mathbf{h}_1(Q')) = (\mathbf{h}_0(Q) + 1,\mathbf{h}_1(Q))$ 
holds by the equation (\ref{euler}). 
If $h\in \mathcal{E}_2$, then $h$ is contained in a cycle $S\subset \mathcal{E}_2$.
Then $s(h)$ and $t(h)$ are also connected in $\mathcal{E}'$, the number of the 
connected components of $Q'$ is equal to that of $Q$. 
Thus we have $(\mathbf{h}_0(Q'),\mathbf{h}_1(Q')) = (\mathbf{h}_0(Q),\mathbf{h}_1(Q) - 1)$.
\end{proof}

Let $L_\alpha$ be a compact connected smooth \sL submanifold of the Calabi-Yau manifold 
$(M,J,\omega,\Omega)$ of ${\rm dim}_{\C}M=m$ for every $\alpha\in\mathcal{V}$.
For every $h\in\mathcal{E}$, suppose $L_{s(h)}$ and $L_{t(h)}$ intersects transversely 
at $p_h\in L_{s(h)} \cap L_{t(h)}$, where $p_h$ is the intersection point of type $1$.
Assume that $p_h\neq p_{h'}$ if $h\neq h'$, and assume that
$\bigcup_{\alpha\in\mathcal{V}}L_\alpha\backslash\{ p_h;h\in\mathcal{E}\}$ is embedded in $M$. 
Let $L_{Q}$ be a differential manifold obtained by taking the connected some 
of $L_{s(h)}$ and $L_{t(h)}$ at $p_h$ for every $h\in \mathcal{E}$.
By Theorem 9.7 of \cite{joyce2003special}, if $(\R_{>0})^\mathcal{E}\cap {\rm Ker}(\partial)$ 
is nonempty, there exists a family of compact smooth \sL submanifolds $\{\tilde{L}_t\}_{0<t<\delta}$ 
which converges to $\bigcup_{\alpha\in\mathcal{V}}$ as $t\to 0$ in the sense of current.
Here, $\tilde{L}_t$ is diffeomorphic to $L_{Q}$. 

Now we can replace the assumption that 
$(\R_{>0})^\mathcal{E}\cap {\rm Ker}(\partial)$ 
is nonempty can be replaced by that $\mathcal{E}$ is covered by cycles.
Consequently, the proof of Theorem \ref{gluing} is completed by the next proposition.

\begin{prop}
If $Q=(\mathcal{V},\mathcal{E},s,t)$ is a connected quiver, then
$L_{Q}$ is diffeomorphic to 
\begin{eqnarray}
L_1\# L_2\# \cdots \# L_A\# N (S^1\times S^{m-1}),\nonumber
\end{eqnarray}
where $\mathcal{V}=\{ 1,\cdots, A\}$ and $N={\rm dim}\ {\rm Ker}(\partial)$, 
and the orientation of each $L_\alpha$ is determined by ${\rm Re}\Omega|_{L_\alpha}$.
\label{topology}
\end{prop}
\begin{proof}
Let $Q=(\mathcal{V},\mathcal{E},s,t)$ be a connected quiver 
and $Q'=(\mathcal{V},\mathcal{E}',s|_{\mathcal{E}'},t|_{\mathcal{E}'})$, 
where $\mathcal{E}'=\mathcal{E}\backslash \{ h\}$.
Let $\mathcal{E}_1,\mathcal{E}_2$ be as in the proof of 
Lemma \ref{quiver2}.

If $h\in \mathcal{E}_1$, then the quiver $Q'$ consists of two 
connected components 
$Q_1 = (\mathcal{W}_1,\mathcal{F}_1,s|_{\mathcal{F}_1},t|_{\mathcal{F}_1})$ 
and $Q_2 = (\mathcal{W}_2,\mathcal{F}_2,s|_{\mathcal{F}_2},t|_{\mathcal{F}_2})$, where 
$\mathcal{V} = \mathcal{W}_1\sqcup \mathcal{W}_2$ and 
$\mathcal{F}_i = \mathcal{E}'\cap (s^{-1}(\mathcal{W}_i)\cup t^{-1}(\mathcal{W}_i))$.
Then we can see that $L_Q = L_{Q_1}\# L_{Q_2}$.

If $h\in \mathcal{E}_2$, then $Q'=(\mathcal{V},\mathcal{E}',s|_{\mathcal{E}'},t|_{\mathcal{E}'})$ is also 
connected, 
hence $L_Q$ is constructed from $L_{Q'}$ in the following way.
Take any distinct points $p_+,p_-\in L_{Q'}$ and their neighborhood 
$B_{p_\pm}\subset L_{Q'}$ so that $B_{p_0}\cap B_{p_1}$ is empty and 
$B_{p_\pm}$ are diffeomorphic to the Euclidean unit ball. 
Now we have a polar coordinate $(r_\pm,\Theta_\pm) \in B_{p_\pm}\backslash \{ p_\pm\}$, where 
$r_\pm \in (0,1)$ is the distance from $p_\pm$, and $\Theta_\pm\in S^{m-1}$.
By taking a diffeomorphism $\psi: (r,\Theta) \mapsto (1-r,\varphi (\Theta))$, we can glue 
$B_{p_+}\backslash \{ p_+\}$ and $B_{p_-}\backslash \{ p_-\}$, then obtain 
$L_Q$. 
Here, $\varphi:S^{m-1} \to S^{m-1}$ is a diffeomorphism 
which reverse the orientation. 
Note that the differentiable structure of $L_Q$ is independent of 
the choice of $p_\pm$, $B_{p_\pm}$ and $\varphi$. 
Therefore we may suppose $p_+$ and $p_-$ is contained in 
an open subset $U\subset L_Q$, where $U= B(0,10)$ and 
$B_{p_\pm} = B(\pm 5, 1)$, respectively.
Here $B(x,r) = \{ x'\in\R^m;\ \| x' - x\| < r\}$.
Then $(U\backslash \{ (1,0),(-1,0)\} )/\psi$ is diffeomorphic to 
$S^1\times S^{m-1}\backslash \{ {\rm pt.}\}$, hence 
$L_Q$ is diffeomorphic to $L_{Q'}\# S^1\times S^{m-1}$.

By repeating these two types of procedures, we finally obtain a quiver 
$Q''=(\mathcal{V},\emptyset,s,t)$, and we have 
$(\mathbf{h}_0(Q''),\mathbf{h}_1(Q'')) = (\#\mathcal{V},0)$.
By counting $(\mathbf{h}_0,\mathbf{h}_1)$ on each step,
it turns out that we have to follow the former procedures $\#\mathcal{V} -1$ times 
and the latter procedures $\mathbf{h}_1(Q)$ times until 
we reach $Q''$.
Therefore we obtain the assertion by considering the procedures 
inductively.
\end{proof}

\section{The construction of compact special Lagrangian submanifolds in $X(u,\lambda)$}\label{sec6}
Here we construct examples of 
compact special Lagrangian submanifolds in $X(u,\lambda)$, 
using Theorem \ref{gluing}.
We construct a one parameter family of 
compact special Lagrangian submanifolds 
which degenerates to the union $\bigcup_i L_i$ of 
some $\sigma_i$-holomorphic Lagrangian submanifolds $L_i$ 
in Subsection \ref{5.1}.

Let $X(u,\lambda)$ be a smooth toric \hK manifold. 
\begin{definition}
{\rm We call $\triangle \subset \ImH \otimes (\mathbf{t}^n)^*$} a $\sigma$-Delzant polytope 
{\rm if it is a compact convex set in 
\begin{eqnarray}
V(q,\sigma):=q + \sigma\otimes (\mathbf{t}^n)^* \nonumber
\end{eqnarray}
for some $q\in \ImH \otimes (\mathbf{t}^n)^*$, 
and the boundary of $\triangle$ in $V(q,\sigma)$ satisfies
\begin{eqnarray}
\partial \triangle = \triangle \cap \bigg( \bigcup_{k=1}^NH_k\bigg).\nonumber
\end{eqnarray}
}
\end{definition}
It is easy to see that $L_{\triangle} := \mu_{\lambda}^{-1}(\triangle)$ is $\sigma$-\hL 
if it is smooth.
Since $T^n$-action is closed on $L_{\triangle}$, 
we may regard $(L_{\triangle},I_{\lambda,1}^\sigma|_{L_{\triangle}})$ as a toric variety, 
equipped with a K\"ahler form $\omega_{\lambda,1}^\sigma|_{L_{\triangle}}$ and a K\"ahler moment map 
$\mu_{\lambda,1}^\sigma: L_{\triangle} \to (\mathbf{t}^n)^*$.
In particular, $L_{\triangle}$ is an oriented manifold whose orientation 
is induced naturally from $I_{\lambda,1}^\sigma$. 
We denote by $\overline{L}_{\triangle}$ the oriented manifold diffeomorphic to 
$L_{\triangle}$ with the opposite orientation.
By the assumption $X(u,\lambda)$ is smooth, $u$ and $\lambda$ satisfies 
$(*1)(*2)$ of Theorem \ref{smooth}, then it is easy to see that 
$\triangle$ is a Delzant polytope in the ordinary sense, 
consequently $L_{\triangle}$ turns out to be a smooth toric variety.

\begin{definition}
{\rm For $\alpha=0,1$, let $\triangle_\alpha$ be a $\sigma(\theta_\alpha)$-Delzant polytope. 
Put \begin{eqnarray}
Q(r) := \{ (t_1,\cdots, t_n)\in (\mathbf{t}^n)^*;
\ t_1\ge 0,\cdots, t_n\ge 0, t_1^2+\cdots + t_n^2 <r^2\}.\nonumber
\end{eqnarray}
Then $\triangle_0$ and $\triangle_1$ are said to be} intersecting standardly with 
angle $\theta$ {\rm if 
$\triangle_0 \cap \triangle_1 = \{ q\}$ and there are $\psi\in GL_n\Z$, $\theta_0\in\R$ and 
sufficiently small $r >0$ such that
\begin{eqnarray}
\psi(\triangle_0 -q)\cap B(r) &=& \sigma(\theta_0)\otimes Q(r),\nonumber\\
\psi(\triangle_1 -q)\cap B(r) &=& \sigma(\theta_0 + \theta)\otimes Q(r),\nonumber
\end{eqnarray}
where $\psi : (\Z^n)^* \to (\Z^n)^*$ extends to $\ImH \otimes (\mathbf{t}^n)^* \to \ImH \otimes (\mathbf{t}^n)^*$
naturally. }
\label{def.intersection}
\end{definition}

For $m\in\Z_{>0}$, let 
\begin{eqnarray}
d_m(l_1,l_2) := \min \{ |l_1 - l_2 + mk|;\ k\in\Z \},\nonumber
\end{eqnarray}
for $l_1,l_2\in \Z$, which induces a distance function on $\Z /m\Z$.

The main result of this article is described as follows.
\begin{thm}
Let $X(u,\lambda)$ be a smooth toric \hK manifold, 
and $\triangle_k$ be a $\sigma(k\pi/n)$-Delzant polytope 
for each $k = 1,\cdots, 2n$.
Assume that $\triangle_k \cap \triangle_l = \emptyset$ if $d_{2n}(k,l)>1$, 
and $\triangle_k$ and $\triangle_{k+1}$ intersecting standardly with 
angle $\pi/n$.
Then there exists a family of compact special Lagrangian submanifolds 
$\{ \tilde{L}_t \}_{0<t<\delta}$ which converges $\bigcup_{k=1}^{2n}L_{\triangle_k}$ 
as $t\to 0$ in the sense of current.
Moreover, $\tilde{L}_t$ is diffeomorphic to 
$L_{\triangle_1}\# 
\overline{L}_{\triangle_2}\# \cdots L_{\triangle_{2n-1}}\# 
\overline{L}_{\triangle_{2n}}\# (S^1\times S^{2n-1})$.
\label{main}
\end{thm}
\begin{proof}
We apply Theorem \ref{gluing}.
By combining Propositions \ref{model} and \ref{angle}, 
we can see that the characterizing angles between $L_{\triangle_k}$ and $L_{\triangle_{k+1}}$ 
are $\frac{\pi}{2n}$ with multiplicity $2n$. 
Then the intersection point $L_{\triangle_k}\cap L_{\triangle_{k+1}}$ is of type $1$.

Next we consider the topology of $\tilde{L}_t$. 
When we take a connected sum, we should determine the orientation of 
$L_{\triangle_k}$ uniformly by the calibration ${\rm Re}\Omega$, 
where $\Omega=(\omega_{\lambda,2} + \sqrt{-1}\omega_{\lambda,3})^n$.
Now $\Omega|_{L_{\triangle_k}} = (-1)^k(\omega_1^{\sigma(k\pi/n)})^n|_{L_{\triangle_k}}$ 
holds, 
therefore $\tilde{L}_t$ is diffeomorphic to 
\begin{eqnarray}
L_{\triangle_1}\# 
\overline{L}_{\triangle_2}\# \cdots L_{\triangle_{2n-1}}\# 
\overline{L}_{\triangle_{2n}}\# (S^1\times S^{2n-1}).\nonumber
\end{eqnarray}
\end{proof}

\subsection{Example $(1)$}\label{5.1}
Let 
\begin{eqnarray}
u = (I_n\ I_n\ \cdots\ I_n) \in {\rm Hom}(\Z^{2n^2}, \Z^n) \nonumber
\end{eqnarray}
and $\lambda = (\lambda_{1,1},\cdots,\lambda_{1,n},\lambda_{2,1},\cdots,\lambda_{2,n},\cdots,\lambda_{2n,1},\cdots,\lambda_{2n,n})$, 
where $I_n$ is the identity matrix.
Then $X(u,\lambda)$ is smooth if $\lambda_{k,\alpha} = \lambda_{l,\alpha}$ holds only if $k=l$.
We assume this condition and that $-\lambda_{k,\alpha} = (0, \rho_{k,\alpha}) \in \{ 0\} \oplus \C$ 
holds for every $k,\alpha$, where $\ImH$ is identified with $\R \oplus \C$.
Moreover we suppose that 
\begin{eqnarray}
\arg (\rho_{k + 1,\alpha} - \rho_{k,\alpha} ) = \theta_0 + \frac{n+1}{n}k\pi
\end{eqnarray}
for some $\theta_0 \in \R$.
Note that $X(u,\lambda)$ is a direct product of multi Eguchi-Hanson spaces.

Next we put $q_k:= -(\lambda_{k,1},\cdots,\lambda_{k,n})\in \ImH\otimes (\mathbf{t}^{n})^*$, 
 and 
\begin{eqnarray}
\square_k &:=& q_k + (0, e^{\sqrt{-1}(\theta_0 + \frac{n+1}{n}k\pi) })\otimes 
\square(r_{k,1},\cdots, r_{k,n})\nonumber\\
&\subset& V(q_k, \sigma(\theta_0 + \frac{n+1}{n}k\pi)), \nonumber
\end{eqnarray}
where $r_{k,\alpha} = |\rho_{k + 1,\alpha} - \rho_{k,\alpha}|$, and 
a hyperrectangle $\square(r_{1},\cdots, r_{n}) \subset (\mathbf{t}^{n})^*\cong \R^n$ is defined by 
\begin{eqnarray}
\square(r_{1},\cdots, r_{n}) := 
\{ (t_1,\cdots, t_n) \in \R^n ;\ 0\le t_1 \le r_1,\ \cdots,\ 0\le t_n \le r_n\}\nonumber
\end{eqnarray}

Let $H_{k,\alpha} = \{ y\in \ImH\otimes (\mathbf{t}^{n})^* ;\ y_k + \lambda_{k,\alpha} = 0 \}$.
Then it is easy to see that $\square_k$ is compact,  convex 
and 
\begin{eqnarray}
\partial \square_k \subset 
\bigcup_{\alpha = 1}^{2n}(H_{k,\alpha}\cup H_{k+1,\alpha}).\nonumber
\end{eqnarray}
Therefore, $\square_k$ is a $\sigma(\theta_0 + (n+1)k\pi/n)$-Delzant polytope
if 
\begin{eqnarray}
\square_k \cap \bigcup_{\alpha, k}H_{k,\alpha} \subset \partial \square_k
\label{inclusion}
\end{eqnarray}
holds.

Next we study the intersection of $\square_{k-1}$ and $\square_k$. 
We can check that $\square_{k-1} \cap \square_k = \{ q_k\}$ and 
every element in $\square_{k-1}$ satisfies
\begin{eqnarray}
&\ & q_{k-1} + (0, e^{\sqrt{-1}(\theta_0 + \frac{n+1}{n}(k-1)\pi )})\otimes (t_1,\cdots, t_n) \nonumber\\
&=& q_{k-1} + (0, e^{\sqrt{-1}(\theta_0 + \frac{n+1}{n}(k-1)\pi )})\otimes (r_{k-1,1},\cdots, r_{k-1,n}) \nonumber\\
&\ & - (0, e^{\sqrt{-1}(\theta_0 + \frac{n+1}{n}(k-1)\pi )})\otimes (r_{k-1,1} - t_1,\cdots, r_{k-1,n}-t_n)\nonumber\\
&=& q_{k-1} + (\lambda_{k-1,1} - \lambda_{k,1},\cdots, \lambda_{k-1,n} - \lambda_{k,n}) \nonumber\\
&\ & 
+ (0, e^{\sqrt{-1}(\theta_0 + \frac{n+1}{n}(k-1)+1)\pi })\otimes (r_{k-1,1} - t_1,\cdots, r_{k-1,n}-t_n)\nonumber\\
&=& q_k + (0, e^{\theta_0 + \sqrt{-1}\frac{(n+1)k -1}{n}\pi })\otimes (r_{k-1,1} - t_1,\cdots, r_{k-1,n}-t_n).\nonumber
\end{eqnarray}
Therefore, $\square_{k-1}$ and $\square_k$ are intersecting standardly with 
angle $\pi/n$.
Of course, the same argument goes well for $\square_{2n}$ and $\square_1$.

To apply Theorem \ref{main}, it suffices to show that 
$\square_k \cap \square_l$ is empty if $d_{2n}(k,l)>1$. 
However, this condition does not hold in general, 
accordingly we need to take $\rho_{k,\alpha}$ well.
Unfortunately, the author cannot find the good criterion 
for $\rho_{k,\alpha}$ satisfying the above condition.
Here we show one example of $\rho_{k,\alpha}$ which satisfies the 
assumption of Theorem \ref{main}.

First of all, take $a_1,\cdots,a_n\in \R$ so that every $a_m$ is larger than 
$1$, and put 
\begin{eqnarray}
\rho_{2m-1} &:=& e^{\sqrt{-1}\frac{2(m-1)}{n}\pi} 
+ a_m(e^{\sqrt{-1}\frac{2m}{n}\pi} - e^{\sqrt{-1}\frac{2(m-1)}{n}\pi}),\nonumber\\
\rho_{2m} &:=& e^{\sqrt{-1}\frac{2(m+1)}{n}\pi} 
+ a_m(e^{\sqrt{-1}\frac{2m}{n}\pi} - e^{\sqrt{-1}\frac{2(m+1)}{n}\pi})\nonumber
\end{eqnarray}
for each $m=1,\cdots, n$.
Denote by $\mathbf{l}_k\subset \C$ the segment connecting 
$\rho_k$ and $\rho_{k+1}$.
Then we can easily see that $\mathbf{l}_{k-1} \cap \mathbf{l}_{k} = \{ \rho_k \}$ and 
\begin{eqnarray}
\arg(\rho_{k+1} -\rho_{k}) = \frac{n-2}{2n}\pi + \frac{n+1}{n}k\pi.\nonumber
\end{eqnarray}
Note that we can regard $k\in\Z /2n \Z$ and $m\in \Z/ n\Z$.

\begin{prop}
Let $\rho_1,\cdots,\rho_{2n}$ be as above. 
If every $a_k-1$ is sufficiently small, then 
$\mathbf{l}_{2m-1} \cap \mathbf{l}_{k}$ are empty 
for all $m=1,\cdots,n$ and $k=1,\cdots,2n$ with $d_{2n}(k,2m-1)>1$.
\label{disjoint}
\end{prop}
\begin{proof}
Let ${\rm Re}:\C \to \R$ be the projection given by taking 
the real part.
It suffices to show that 
${\rm Re}(\mathbf{l}_{2m-1}e^{-\sqrt{-1}\frac{2m}{n}\pi}) \cap 
{\rm Re}(\mathbf{l}_{k}e^{-\sqrt{-1}\frac{2m}{n}\pi})$ 
is empty under the given assumptions.
Let $\rho_{2m-1} + t(\rho_{2m}-\rho_{2m-1})\in \mathbf{l}_{2m-1}$.
Then we can check that
\begin{eqnarray}
{\rm Re}(\rho_{2m-1}e^{-\sqrt{-1}\frac{2m}{n}\pi} 
+ t(\rho_{2m}-\rho_{2m-1})e^{-\sqrt{-1}\frac{2m}{n}\pi}) 
= (1-a_m)\cos \frac{2\pi}{n} + a_m,\nonumber
\end{eqnarray}
which implies ${\rm Re}(\mathbf{l}_{2m-1}e^{-\sqrt{-1}\frac{2m}{n}\pi}) 
= \{ -(a_m-1)\cos \frac{2\pi}{n} + a_m\}$.
If we can see that 
\begin{eqnarray}
{\rm Re}(\rho_k e^{-\sqrt{-1}\frac{2m}{n}\pi}) < -(a_m-1)\cos \frac{2\pi}{n} + a_m\label{ineq1}
\end{eqnarray}
for all $k$, we have the assertion.
Since 
\begin{eqnarray}
{\rm Re}(\rho_{2l} e^{-\sqrt{-1}\frac{2m}{n}\pi}) &=& -(a_{l}-1)\cos(\frac{2(1 + l-m)}{n}\pi) \nonumber\\
&\ &\quad + a_{l}\cos(\frac{2(l-m)}{n}\pi),\nonumber\\
{\rm Re}(\rho_{2l'-1} e^{-\sqrt{-1}\frac{2m}{n}\pi}) &=& -(a_{l'} -1)\cos(\frac{2(1 - l'+m)}{n}\pi) \nonumber\\
&\ &\quad 
+ a_{l'}\cos(\frac{2(l'-m)}{n}\pi)\nonumber
\end{eqnarray}
and $d_{2n}(2l,2m)>1$, $d_{2n}(2l'-1,2m)>1$ holds,
we have $\cos(2(l-m)\pi /n) \le \cos(2\pi / n)$ and 
$\cos(2(l'-m)\pi /n) \le \cos(2\pi / n)$.
By using the inequality $\cos(\frac{2(1 + l-m)}{n}\pi) \ge -1$ and 
$\cos(\frac{2(1 - l'+m)}{n}\pi) \ge -1$, we obtain
\begin{eqnarray}
{\rm Re}(\rho_{2l} e^{-\sqrt{-1}\frac{2m}{n}\pi}) &\le& (a_{l}-1) 
+ a_{l}\cos\frac{2\pi}{n} \nonumber\\
&=&(a_{l}-1)(1+ \cos\frac{2\pi}{n}) + \cos\frac{2\pi}{n},\nonumber\\
{\rm Re}(\rho_{2l'-1} e^{-\sqrt{-1}\frac{2m}{n}\pi}) &\le& (a_{l'} -1) 
+ a_{l'}\cos\frac{2\pi}{n}\nonumber\\
&=& (a_{l'}-1)(1+ \cos\frac{2\pi}{n}) + \cos\frac{2\pi}{n}\nonumber
\end{eqnarray}
Now, if we assume $a_l-1<(1-\cos\frac{2\pi}{n})/(1+ \cos\frac{2\pi}{n})$, 
then the left-hand-side of (\ref{ineq1}) is less than $1$.
Since
\begin{eqnarray}
-(a_m-1)\cos \frac{2\pi}{n} + a_m = (a_m-1)(1- \cos \frac{2\pi}{n}) + 1,\nonumber
\end{eqnarray}
the right-hand-side of (\ref{ineq1}) is always larger than $1$ and 
we obtain the inequality (\ref{ineq1}).
\end{proof}

Now, divide $\{ 1,\cdots, n\}$ into two nonempty sets 
\begin{eqnarray}
\{ 1,\cdots, n\} = A_+\sqcup A_-, \nonumber
\end{eqnarray}
and define $\rho_{k,\alpha}$ by $\rho_{k,\alpha} = \rho_k$ 
if $\alpha \in A_+$, and 
$\rho_{k,\alpha} = \rho_ke^{\sqrt{-1}\pi/n}$ if $\alpha \in A_-$.
Here, we suppose $a_k-1$ are sufficiently small so that 
Proposition \ref{disjoint} holds.
Denote by $\mathbf{l}_{k,\alpha}$ the segment in $\ImH$ connecting 
$(0,\rho_{k,\alpha})$ and $(0,\rho_{k+1,\alpha})$.
Then we can see that
\begin{eqnarray}
\square_k = q_k + \mathbf{l}_{k,1}\times \cdots \times \mathbf{l}_{k,n},\nonumber
\end{eqnarray}
and $\square_k$ satisfies (\ref{inclusion}).

\begin{prop}
Let $\square_1,\cdots,\square_{2n}$ be as above. 
Then $\square_k \cap \square_l$ is empty if $d_{2n}(k,l)>1$.
\end{prop}
\begin{proof}
Suppose there is an element $\hat{x}\in \square_k \cap \square_l$.
Then $\hat{x}$ can be written as $\hat{x}=(0,x)$ for some 
$x = (x_1,\cdots,x_n)\in \C^n$, and 
\begin{eqnarray}
x_\alpha = \rho_{k,\alpha} + t_{k,\alpha} (\rho_{k+1,\alpha} - \rho_{k,\alpha}) 
=  \rho_{l,\alpha} + t_{l,\alpha} (\rho_{l+1,\alpha} - \rho_{l,\alpha})
\label{lineareq}
\end{eqnarray}
holds for some $0\le t_{k,\alpha},t_{l,\alpha}\le 1$.
Now, assume that $l$ is odd.
Then (\ref{lineareq}) has no solution for $\alpha \in A_+$ by Proposition \ref{disjoint}.
Similarly, if  $l$ is supposed to be even, then (\ref{lineareq}) has no solution for $\alpha \in A_-$.
Hence $\square_k \cap \square_l$ should be empty.
\end{proof}
Since $L_{\square_k} = (\mathbb{P}^1)^n$, 
and there is an orientation preserving diffeomorphism 
between $(\mathbb{P}^1)^n$ and $(\overline{\mathbb{P}^1})^n$, 
we obtain the following example.
\begin{thm}
Let $X(u,\lambda)$ be as above.
Then there exists a compact smooth \sL 
submanifold diffeomorphic to 
\begin{eqnarray}
2n(\mathbb{P}^1)^n \# (S^1\times S^{2n-1})\nonumber
\end{eqnarray}
embedded in $X(u,\lambda)$.
\end{thm}

\subsection{Example $(2)$}\label{5.2}
Here we construct one more example in 
an $8$ dimensional toric \hK manifolds.

Let
\[ u := 
\left (
\begin{array}{ccccc}
1 & 1 & 0 & 1 & 0 \\
1 & 0 & 1 & 0 & 1
\end{array}
\right ) \in {\rm Hom}(\Z^5, \Z^2),
\]
and $\lambda=(\lambda_0,\cdots,\lambda_4)\in \ImH\otimes (\mathbf{t}^5)^*$.
Put 
\begin{eqnarray}
q_1:=-(\lambda_1,\lambda_2),\ q_2:=-(\lambda_3,\lambda_2),\ q_3:=-(\lambda_3,\lambda_4),\ q_4:=-(\lambda_1,\lambda_4)\nonumber
\end{eqnarray}
and $\triangle_k := q_k + \tau_k\otimes \triangle$ 
for $k=1,\cdots,4$, where 
\begin{eqnarray}
\tau_1 &:=& \lambda_1 + \lambda_2 -\lambda_0, \nonumber\\
\tau_2 &:=& \lambda_3 + \lambda_2 -\lambda_0, \nonumber\\
\tau_3 &:=& \lambda_3 + \lambda_4 -\lambda_0, \nonumber\\
\tau_4 &:=& \lambda_1 + \lambda_4 -\lambda_0, \nonumber
\end{eqnarray}
and 
\begin{eqnarray}
\triangle := \{ (t_1,t_2) \in (\mathbf{t}^2)^*\cong \R^2;
\ t_1\ge 0,\ t_2\ge 0,\ t_1+t_2 \le 1\}.\nonumber
\end{eqnarray}
If we assume that $\tau_1 = (0,\sqrt{-1}r_1)$, 
$\tau_2 = (0,-r_2)$, $\tau_3 = (0,-\sqrt{-1}r_1)$, 
$\tau_4 = (0,r_2)\in \R \oplus \C$, where $r_1,r_2>0$, then $\triangle_k$ is a 
$\sigma(k\pi/2)$-Delzant polytope. 
For example, put $\lambda_0 = \lambda_1 = 0$, 
$\lambda_2 = (0,\sqrt{-1}r_1)$, $\lambda_3 = (0,-r_2 - \sqrt{-1}r_1)$ 
and $\lambda_4 = (0,r_2)$.
\begin{prop}
Under the above setting, $\triangle_k$ and $\triangle_{k+1}$ are 
intersecting standardly with angle $\pi/2$ for every 
$k=1,\cdots,4$. 
Here we suppose $\triangle_5 = \triangle_1$.
\end{prop}

\begin{proof}
We check the case of $k=1$, because other cases can be shown 
similarly. 
Let $q_k + \tau_k\otimes (t_1,t_2) \in \triangle_k$. 
Then we have 
\begin{eqnarray}
q_1 + \tau_1\otimes (t_1,t_2) 
&=& q_1 + \tau_1\otimes (1,0) + \tau_1\otimes (t_1-1,t_2)\nonumber\\
&=& (\lambda_2-\lambda_0,-\lambda_2) + \sigma(\frac{\pi}{2})
\otimes r_1(t_1-1,t_2),\nonumber\\
q_2 + \tau_2\otimes (t_1,t_2) 
&=& q_2 + \tau_2\otimes (1,0) + \tau_2\otimes (t_1-1,t_2)\nonumber\\
&=& (\lambda_2-\lambda_0,-\lambda_2) + \sigma(\pi)\otimes r_2(t_1-1,t_2),\nonumber
\end{eqnarray}
therefore $\triangle_1$ and $\triangle_2$ are 
intersecting standardly with angle $\pi/2$.
Note that we have to take 
\[ \psi = 
\left (
\begin{array}{cc}
-1 & 0\\
0 & 1
\end{array}
\right )
\]
in Definition \ref{def.intersection}, since 
$t_1 - 1$ is nonpositive in this case.
\end{proof}
\begin{prop}
Under the above setting, $\triangle_1 \cap \triangle_3$ 
and $\triangle_2 \cap \triangle_4$ are empty sets.
\end{prop}
\begin{proof}
Every $x\in \triangle_1 \cap \triangle_3$ can be written as 
\begin{eqnarray}
x = q_1 + \tau_1\otimes (t_1,t_2) 
= q_3 + \tau_3\otimes (s_1,s_2)\nonumber
\end{eqnarray}
for some $0\le t_1,t_2,s_1,s_2\le 1$.
The first component of the above equation gives 
\begin{eqnarray}
t_1\tau_1 - s_1\tau_3 = \lambda_1 - \lambda_3 = \tau_1 - \tau_2.\nonumber
\end{eqnarray}
By substituting $\tau_1 = (0,\sqrt{-1}r_1)$, 
$\tau_2 = (0,-r_2)$ and 
$\tau_3 = (0,-\sqrt{-1}r_1)$, 
we obtain 
\begin{eqnarray}
(t_1 + s_1)\sqrt{-1}r_1 = r_2 + \sqrt{-1}r_1,\nonumber
\end{eqnarray}
which has no solution $(t_1,s_1) \in \R^2$.
$\triangle_2 \cap \triangle_4 = \emptyset$ also follows by 
the same argument.
\end{proof}
Thus we obtain the following example.
\begin{thm}
Let $X(u,\lambda)$ be as above.
Then there exists a compact smooth \sL 
submanifold diffeomorphic to 
$2\mathbb{P}^2 \# 2\overline{\mathbb{P}^2} \# (S^1\times S^3)$
embedded in $X(u,\lambda)$.
\end{thm}

\subsection{Example $(3)$}\label{5.3}
We can describe a generalization of Theorem \ref{main} 
in the more complicated situation.

\begin{thm}
Let $(\mathcal{V},\mathcal{E},s,t)$ be a quiver, 
$X(u,\lambda)$ be a smooth toric \hK manifold, and 
$\{\triangle_k \}_{k\in \mathcal{V}}$ be a family of 
subsets of $\ImH \otimes (\mathbf{t}^n)^*$.
Assume that every $\triangle_k$ is a 
$\sigma(\theta_k)$-Delzant polytope for some $\theta_k\in\R$, 
$\triangle_{s(h)}$ and $\triangle_{t(h)}$ intersecting standardly with 
angle $\pi/n$ if $h \in\mathcal{E}$, 
otherwise $\triangle_{k_1} \cap \triangle_{k_2} = \emptyset$ or $k_1=k_2$.
Moreover, suppose that $\mathcal{E}$ is covered by cycles.
Then there exists a family of compact \sL submanifolds $\{ \tilde{L}_t\}_{0<t<\delta}$ 
which converges to $\bigcup_{k\in \mathcal{V}} L_{\triangle_k}$ in the sense of current.
\label{graph}
\end{thm}
\begin{proof}
The proof is same as that of Theorem \ref{main}. 
\end{proof}
Fix positive real numbers $a,b,c,a_m$ for $m=1,\cdots,N$ so that 
$0<a_1<a_2<\cdots <a_N$.
Let 
\[ u = 
\left (
\begin{array}{cccccccc}
1 & 1 & 1 & 1 & 0 & 0 & \cdots & 0 \\
0 & 0 & 0 & 0 & 1 & 1 & \cdots & 1 
\end{array}
\right ) \in {\rm Hom}(\Z^{2N+6}, \Z^2)
\]
and 
\begin{eqnarray}
\lambda = (\lambda_{-3}, \lambda_{-2}, \lambda_{-1}, \lambda_0, \lambda_1,\cdots,\lambda_{2N+2}) 
\in \ImH\otimes (\mathbf{t}^{2N+6})^* ,\nonumber
\end{eqnarray}
where $-\lambda_0=0$, $-\lambda_{-1}=(0,\sqrt{-1}b)$,  $-\lambda_{-2}=(0,a+\sqrt{-1}b)$, 
$-\lambda_{-3}=(0,a)$, $-\lambda_{2m+1}=(0,a_m + \sqrt{-1}c)$ and $-\lambda_{2m+2}=(0,a_m)$
for $m=0,1,\cdots,N$.
Here, we put $a_0=0$.
Then $X(u,\lambda)$ is smooth and become the direct product $X(u',\lambda')\times X(u'',\lambda'')$ 
where $u'=(1,1,1,1) \in {\rm Hom}(\Z^4,\Z)$, $u''=(1,\cdots,1) \in {\rm Hom}(\Z^{2N+2},\Z)$, 
$\lambda' =(\lambda_{-3}, \lambda_{-2}, \lambda_{-1}, \lambda_0)$ and 
$\lambda''=(\lambda_1,\cdots,\lambda_{2N+2})$.
Denote by $[p,q]\subset \ImH$ the segment connecting $p,q\in \ImH$, 
and put 
$\mathbf{A}_{-}:=[-\lambda_0,-\lambda_{-1}]$, $\mathbf{A}_{+}:=[-\lambda_{-2},-\lambda_{-3}]$, 
$\mathbf{S}_{+}:=[-\lambda_{-1},-\lambda_{-2}]$, $\mathbf{S}_{-}:=[-\lambda_{-3},-\lambda_0]$, 
$\mathbf{A}_{m}:=[-\lambda_{2m+1},-\lambda_{2m+2}]$ for $m=0,1,\cdots,N$, 
$\mathbf{S}_{+,m}:=[-\lambda_{2m-1},-\lambda_{2m+1}]$ and 
$\mathbf{S}_{-,m}:=[-\lambda_{2m},-\lambda_{2m+2}]$ for $m=1,\cdots, N$.

Let 
\begin{eqnarray}
\square_{2l,1} &:=& \mathbf{A}_{-}\times \mathbf{A}_{2l},\nonumber\\
\square_{2l,2} &:=& \mathbf{S}_{+}\times \mathbf{S}_{+,2l+1}, \nonumber\\
\square_{2l,3} &:=& \mathbf{A}_{+}\times \mathbf{A}_{2l+1}, \nonumber\\
\square_{2l,4} &:=& \mathbf{S}_{-}\times \mathbf{S}_{-,2l+1}\nonumber
\end{eqnarray}
for $l=0,1,\cdots, [(N-1)/2]$, and 
\begin{eqnarray}
\square_{2l-1,1} &:=& \mathbf{A}_{-}\times \mathbf{A}_{2l}, \nonumber\\
\square_{2l-1,2} &:=& \mathbf{S}_{+}\times \mathbf{S}_{-,2l}, \nonumber\\
\square_{2l-1,3} &:=& \mathbf{A}_{+}\times \mathbf{A}_{2l-1}, \nonumber\\
\square_{2l-1,4} &:=& \mathbf{S}_{-}\times  \mathbf{S}_{+,2l}\nonumber
\end{eqnarray}
for $l=1,\cdots, [N/2]$.
Then $\square_{m,j}$ is a $\sigma(j\pi/2)$-\hL submanifold 
satisfying
$\square_{2l-1,1}=\square_{2l,1}$ and $\square_{2l,3}=\square_{2l+1,3}$, 
moreover $\square_{m,j}$ and $\square_{m,j+1}$ intersect standardly with angle $\pi/2$
for $j=1,2,3,4$, where we put $\square_{m,5}=\square_{m,1}$.
Otherwise, $\square_{m,j} \cap \square_{m',j'}$ is empty. 

Now let
\begin{eqnarray}
\mathcal{V} &:=& (\{ 0,1,\cdots,N\}\times \{ 1,2,3,4\})/\sim,\nonumber
\end{eqnarray}
where $\sim$ is defined by $(2l-1,1)\sim (2l,1)$ and $(2l,3)\sim (2l+1,3)$.
We denote by $[m,j]\in \mathcal{V}$ the equivalence class represented by 
$(m,j)$.
Put 
\begin{eqnarray}
\mathcal{E} &:=& \{ [m,j]\to [m,j+1], [m,4]\to [m,1];\ m=0,\cdots, N,\ j=1,2,3\},\nonumber
\end{eqnarray}
where $x\to y$ means the directed edge whose source is $x$ and the target is $y$.
Then we obtain a quiver $(\mathcal{V},\mathcal{E}s,t)$ 
and it is easy to see that $\mathcal{E}$ 
is covered by cycles.
By this setting, we can see that $\{ \square_{m,j}\}_{[m,j]\in\mathcal{V}}$ satisfies the assumption 
of Theorem \ref{graph}, and we have the following result.
\begin{thm}
Let $X(u,\lambda)$ be as above.
Then there exists a compact smooth \sL 
submanifold diffeomorphic to 
\begin{eqnarray}
(3N+1)(\mathbb{P}^1)^2 \# N(S^1\times S^3)\nonumber
\end{eqnarray}
embedded in $X(u,\lambda)$.
\end{thm}

\section{Obstruction}\label{sec7}
Here we introduce obstructions for the existence of 
\hL and \sL submanifolds in \hK manifolds.
Throughout of this section, let 
$(M^{4n},g,I_1,I_2,I_3)$ be a \hK manifold.
\begin{prop}
Let $L\subset M$ be a \sL submanifold, 
and also a $\sigma$-\hL submanifold for some $\sigma\in S^2$. 
Then $\sigma = \sigma(k\pi/n)$ for some $k = 1,\cdots, 2n$.
\label{6.1}
\end{prop}
\begin{proof}
By decomposing $\R^3$ into $\R\sigma$ and its orthogonal complement, 
we have 
\begin{eqnarray}
(1,0,0) = p\sigma + q\tau\nonumber
\end{eqnarray}
for some $p,q\in \R$ and $\tau\in S^2$, where $\tau$ is orthogonal to $\sigma$.
Then we have $\omega_1= p\omega^\sigma + q\omega^\tau$ and 
\begin{eqnarray}
0= \omega_1|_L= p\omega^\sigma|_L + q\omega^\tau|_L = p\omega^\sigma|_L,\nonumber
\end{eqnarray}
since $L$ is a \sL and $\sigma$-\hL submanifold.
Hence $p$ should be $0$ since $\omega^\sigma$ is non-degenerate on $L$.
Thus we have $(1,0,0) = q\tau$, which means that $\sigma$ is orthogonal to 
$(1,0,0)$. 
Then we may write $\sigma = \sigma(\theta)$ for some $\theta\in \R$.
By the condition ${\rm Im}(\omega_2 + \sqrt{-1}\omega_3)^n|_L = 0$, 
we obtain $\theta = k\pi/n$ for some $k=1,\cdots, 2n$.
\end{proof}

\begin{prop}
Let $L$ be a compact $\sigma(\theta)$-\hL submanifold in $M$ 
for some $\theta$, 
and the orientation of $L$ be determined by 
$\omega^{\sigma(\theta)}_1$.
Then the pairing of the de Rham cohomology class $[\omega_2+\sqrt{-1}\omega_3]^n$ 
and the homology class $[L]\in H_{2n}(M,\Z)$ is given by
\begin{eqnarray}
\langle [\omega_2+\sqrt{-1}\omega_3]^n, [L]\rangle = e^{\sqrt{-1}n \theta} V, \nonumber
\end{eqnarray}
where $V(L) >0$ is the volume of $L$.
\label{6.2}
\end{prop}
\begin{proof}
Since $L$ is $\sigma(\theta)$-holomorphic Lagrangian, 
we have 
\begin{eqnarray}
\omega_1|_L = \omega^{\hat{\sigma}(\theta)}|_L=0,\nonumber
\end{eqnarray}
where $\hat{\sigma}(\theta) = (0,-\sin\theta,\cos\theta) \in S^2$.
Then we obtain
\begin{eqnarray}
\langle [\omega_2+\sqrt{-1}\omega_3]^n, [L]\rangle &=& \int_{L}(\omega_2+\sqrt{-1}\omega_3)^n\nonumber\\
&=& \int_{L}e^{\sqrt{-1}n \theta}\{ e^{-\sqrt{-1} \theta}(\omega_2+\sqrt{-1}\omega_3) \}^n\nonumber\\
&=& e^{\sqrt{-1}n \theta}\int_{L}(\omega^{\sigma(\theta)}+\sqrt{-1}\omega^{\hat{\sigma}(\theta)})^n \nonumber\\
&=& e^{\sqrt{-1}n \theta}\int_{L}(\omega^{\sigma(\theta)})^n
= e^{\sqrt{-1}n \theta} V(L). \nonumber
\end{eqnarray}
\end{proof}
Let $L_1,L_2,\cdots, L_A$ be compact smooth submanifolds 
of dimension $2n$ embedded in $M$. 
Assume that each $L_\alpha$ is a 
$\sigma(\theta_\alpha)$-\hL submanifold for some $\theta_\alpha\in \frac{\pi}{n} \Z$, 
and the orientation of $L_\alpha$ is determined by 
$\omega^{\sigma(\theta_\alpha)}$.
Put $\varepsilon_\alpha =1$ if $\frac{n}{\pi}\theta_\alpha$ is even, 
and $\varepsilon_\alpha =-1$ if $\frac{n}{\pi}\theta_\alpha$ is odd.
\begin{prop}
Under the above setting, 
assume that there exists a compact smooth $\sigma(\theta)$-\hL submanifold $L$ in 
the homology class $\sum_{\alpha=1}^A \varepsilon_\alpha [L_\alpha]$ for some $\theta\in \R$.
Then $\{ \theta_1,\theta_2,\cdots,\theta_A\}$ is contained in $\theta + \pi\Z$.
\label{homology}
\end{prop}
\begin{proof}
Since $L_\alpha$ is a 
$\sigma(\theta_\alpha)$-\hL submanifold, 
we have 
\begin{eqnarray}
\langle [\omega_2+\sqrt{-1}\omega_3]^n, \sum_{\alpha=1}^A \varepsilon_\alpha [L_\alpha ] \rangle 
&=& \sum_{\alpha=1}^A \varepsilon_\alpha e^{\sqrt{-1}n \theta_\alpha} V(L_\alpha)\nonumber\\
&=& \sum_{\alpha=1}^A \varepsilon_\alpha^2 V(L_\alpha) = \sum_{\alpha=1}^A V(L_\alpha).
 \label{eq5}
\end{eqnarray}
by Proposition \ref{6.2}.
Since $L$ is a compact smooth $\sigma(\theta)$-\hL submanifold, 
$\omega_1|_L = \omega^{\hat{\sigma}(\theta)}|_L=0$ holds, 
where we put ${\hat{\sigma}(\theta)}$ as in the proof of Proposition \ref{6.2}.
Therefore we obtain 
\begin{eqnarray}
\langle [\omega_2+\sqrt{-1}\omega_3]^n, [L] \rangle 
= e^{\sqrt{-1}n\theta} \langle [\omega^{\sigma(\theta)}]^n, [L] \rangle 
= e^{\sqrt{-1}n\theta} \langle [\omega^{\sigma(\theta)}]^n, [L] \rangle \label{eq6}
\end{eqnarray}
by the same computation in the proof of Proposition \ref{6.2}.
Then by combining (\ref{eq5})(\ref{eq6}), $\theta$ is given by $\theta = k\pi/n$ for an 
integer $k=1,\cdots,2n$.
Now we have
\begin{eqnarray}
\omega^{\sigma(\theta)} 
&=& {\rm Re}(e^{-\sqrt{-1}\theta}(\omega_2+\sqrt{-1}\omega_3)) \nonumber\\
&=& {\rm Re}(e^{-\sqrt{-1}(\theta - \theta_\alpha)}
e^{-\sqrt{-1}\theta_\alpha}(\omega_2+\sqrt{-1}\omega_3 ))\nonumber\\
&=& {\rm Re}(e^{-\sqrt{-1}(\theta - \theta_\alpha)}
(\omega^{\sigma(\theta_\alpha)}+\sqrt{-1}\omega^{\hat{\sigma}(\theta_\alpha)}))\nonumber\\
&=& \cos (\theta - \theta_\alpha)\omega^{\sigma(\theta_\alpha)}
+ \sin (\theta - \theta_\alpha)\omega^{\hat{\sigma}(\theta_\alpha)}\nonumber
\end{eqnarray}
and $\omega^{\hat{\sigma}(\theta_\alpha)}|_{L_\alpha} = 0$,
we obtain
\begin{eqnarray}
\langle [\omega^{\sigma(\theta)}]^n, [L] \rangle 
&=& \sum_{\alpha=1}^A \varepsilon_\alpha \langle [\omega^{\sigma(\theta)}]^n, [L_\alpha ] \rangle \nonumber\\
&=& \sum_{\alpha=1}^A \varepsilon_\alpha \cos^n (\theta - \theta_\alpha)\langle [\omega^{\sigma(\theta_\alpha)}
]^n, [L_\alpha ] \rangle \nonumber\\
&=& \sum_{\alpha=1}^A \varepsilon_\alpha \cos^n (\theta - \theta_\alpha) V(L_\alpha )
\label{eq7}
\end{eqnarray}
By combining (\ref{eq5})(\ref{eq6})(\ref{eq7}) and putting $\theta=k\pi/n$, we obtain 
\begin{eqnarray}
\sum_{\alpha=1}^A V(L_\alpha) 
= (-1)^k\sum_{\alpha=1}^A \varepsilon_\alpha \cos^n \bigg(\frac{k\pi}{n} - \theta_\alpha\bigg) V(L_\alpha). \nonumber
\end{eqnarray}
Next we put $\theta_\alpha = k_\alpha\pi/n$ for $k_\alpha=1,\cdots,2n$. 
Then $\varepsilon_\alpha = (-1)^{k_\alpha}$ and 
\begin{eqnarray}
\sum_{\alpha=1}^A V(L_\alpha) 
= \sum_{\alpha=1}^A (-1)^{k-k_\alpha} \cos^n \bigg(\frac{k-k_\alpha}{n}\pi\bigg) V(L_\alpha) \nonumber
\end{eqnarray}
holds.
Since every $V(L_\alpha)$ is positive, we obtain 
\begin{eqnarray}
(-1)^{k-k_\alpha} \cos^n \bigg(\frac{k-k_\alpha}{n}\pi\bigg) = 1,
\label{eq8}
\end{eqnarray}
then $k-k_\alpha$ should be contained in $n\Z$. 
If $k-k_\alpha = nl$ for some $l\in \Z$, then $\cos (\frac{k-k_\alpha}{n}\pi) = \cos l\pi = (-1)^l$ 
holds, which gives 
\begin{eqnarray}
(-1)^{k-k_\alpha} \cos^n \bigg(\frac{k-k_\alpha}{n}\pi\bigg) 
= (-1)^{nl} (-1)^{nl} = 1.\nonumber
\end{eqnarray}
Thus the assertion follows since (\ref{eq8}) holds if and only if $k-k_\alpha \in n\Z$.
\end{proof}
\begin{cor}
Under the assumption of Theorem \ref{graph}, let $\mathcal{E}$ be nonempty.
Then the \sL submanifolds $\tilde{L}_t$ obtained in Theorem \ref{graph} is not 
$\sigma$-\hL submanifold for any $\sigma\in S^2$.
\end{cor}
\begin{proof}
First of all, note that $\tilde{L}_t$ is contains in 
$\sum_{k\in\mathcal{V}}\varepsilon_{k}[L_{\triangle_k}]$.
Let $k_1 \to k_2\in \mathcal{E}$.
Then $\triangle_{k_1}$ and $\triangle_{k_2}$ are intersecting 
standardly with angle $\pi/n$, 
hence we have $\theta_{k_2} = \theta_{k_1} + \pi/n$, 
which implies that $\{\theta_k;\ k\in\mathcal{V}\}$ contains 
$\theta_{k_1}$ and $\theta_{k_1} + \pi/n$.
Thus $\{\theta_k;\ k\in\mathcal{V}\}$ never be 
contained in $\theta+\pi\Z$ for any $\theta$ since $n>1$.
By Propositions \ref{6.1} and \ref{homology}, $\tilde{L}_t$ never becomes 
$\sigma$-\hL submanifold for any $\sigma\in S^2$.
\end{proof}

Keio University, 3-14-1 Hiyoshi, Kohoku, Yokohama 223-8522, Japan,\\
hattori@math.keio.ac.jp
\end{document}